\documentclass{amsart}
\usepackage{amscd, amsmath,amssymb,amsfonts}
\usepackage[cmtip, all]{xy}

\newtheorem{thm}{Theorem}[section]
\newtheorem{prop}[thm]{Proposition}
\newtheorem{lem}[thm]{Lemma}
\newtheorem{cor}[thm]{Corollary}

\renewcommand{\theclaim}{\kern-3pt}

\theoremstyle{definition}
\newtheorem{definition}[thm]{Definition}

\theoremstyle{remark}
\newtheorem{rem}[thm]{Remark}

\newtheorem{ex}[thm]{Example}

\numberwithin{equation}{section}

\newcommand{\sC}{{\mathcal C}}

\newcommand{\sE}{{\mathcal E}}
\newcommand{\sF}{{\mathcal F}}
\newcommand{\sG}{{\mathcal G}}
\newcommand{\sH}{{\mathcal H}}

\newcommand{\sO}{{\mathcal O}}

\newcommand{\sS}{{\mathcal S}}
\newcommand{\sT}{{\mathcal T}}

\newcommand{\sV}{{\mathcal V}}

\newcommand{\sX}{{\mathcal X}}

\newcommand{\A}{{\mathbb A}}

\newcommand{\G}{{\mathbb G}}

\renewcommand{\P}{{\mathbb P}}
\newcommand{\Q}{{\mathbb Q}}

\newcommand{\mS}{{\mathbb S}}

\newcommand{\Z}{{\mathbb Z}}

\renewcommand{\L}{{\mathbb L}}

\renewcommand{\phi}{\varphi}

\newcommand{\Div}{{\rm Div}}
\newcommand{\Hom}{{\rm Hom}}

\newcommand{\Spec}{{\rm Spec \,}}

\newcommand{\0}{\emptyset}

\newcommand{\id}{{\operatorname{id}}}

\newcommand{\Sch}{{\operatorname{\mathbf{Sch}}}}

\newcommand{\del}{\partial}

\renewcommand{\max}{{\operatorname{\rm max}}}

\newcommand{\Spt}{{\mathbf{Spt}}}
\newcommand{\Spc}{{\mathbf{Spc}}}
\newcommand{\Sm}{{\mathbf{Sm}}}

\newcommand{\hocolim}{\mathop{{\rm hocolim}}}
\newcommand{\colim}{\operatornamewithlimits{\varinjlim}}

\newcommand{\GL}{{\operatorname{\rm GL}}}

\newcommand{\fin}{{\operatorname{\rm fin}}}
\newcommand{\SH}{{\operatorname{\sS\sH}}}

\newcommand{\eff}{{\mathop{eff}}}

\newcommand{\DM}{{DM}}

\newcommand{\Nis}{{\operatorname{Nis}}}

\newcommand{\ds}{{/\kern-3pt/}}

\newcommand{\Th}{{\mathop{\rm{Th}}}}

\newcommand{\Gr}{{\mathop{\rm{Grass}}}}

\newcommand{\Deg}{{\mathop{\rm{deg}}}}

\newcommand{\EM}{{\mathop{EM}}}

\newcommand{\SP}{{\mathbf{SP}}}
\newcommand{\gr}{\mathbf{Gr}}

\newcommand{\MGL}{{\operatorname{MGL}}}

\newcommand{\lci}{\text{l.\,c.\,i.\,}}

\newcommand{\GrAb}{\mathbf{GrAb}}

\newcommand{\ch}{\rm{ch}}

\newcommand{\ph}{\text{ph}}

\begin{document}
\title{Motivic Landweber exact theories and their effective covers}
\author{Marc Levine}
\address{
Universit\"at Duisburg-Essen\\
Fakult\"at Mathematik\\
Thea-Leymann-Stra{\ss}e 9\\
45127 Essen\\
Germany}
\email{marc.levine@uni-due.de} 


 
 \thanks{The 
author  thanks the Humboldt Foundation for support through the Humboldt Professorship.}

\begin{abstract} Let $k$ be a field of characteristic zero and let $(F,R)$ be a Landweber exact formal group law. We consider a Landweber exact $T$-spectra $\sE:=R\otimes_\L\MGL$ and its effective cover $f_0\sE\to \sE$ with respect to Voevodsky's slice tower. The coefficient ring $R_0$ of $f_0\sE$ is the subring of $R$ consisting of elements of $R$ of non-positive degree; the power series $F\in R[[u,v]]$ has coefficients in $R_0$ although $(F,R_0)$ is not necessarily Landweber exact. We show that the geometric part  $X\mapsto f_0\sE^*(X):=(f_0\sE)^{2*,*}(X)$  of $f_0\sE$  is canonically isomorphic to the oriented cohomology theory $X\mapsto R_0 \otimes_\L \Omega^*(X)$, where   $\Omega^*$ is the theory of {\em algebraic cobordism} as defined in \cite{LevineMorel}. This recovers results of Dai-Levine \cite{DaiLevine} as the special case of algebraic $K$-theory and its effective cover, connective algebraic $K$-theory.
\end{abstract}

\maketitle

\setcounter{tocdepth}{1}
\tableofcontents

\section*{Introduction} Let $S$ be a fixed base-scheme, $\Sm/S$ the category of smooth quasi-projective $S$-schemes, and $\SH(S)$ the motivic stable homotopy category of $T$-spectra. In this paper we consider two types of  cohomology theories which carry the designation ``oriented''. The first type are those bi-graded theories on $\Sm/S$, $X\mapsto \sE^{*,*}(X)$  represented by a (weak) commutative ring $T$-spectrum $\sE\in \SH(S)$ with an orientation $c$ in the reduced cohomology $\tilde{\sE}^{2,1}(\P^\infty)$. The second are the oriented theories in the sense of \cite[definition 1.1.1]{LevineMorel}, that is, contravariant functors $X\mapsto  A^*(X)$ from $\Sm/S$ to commutative graded rings, together with pushforward maps $f_*$ for projective morphisms $f:Y\to X$, satisfying a number of functorialities and compatibilities. We will refer to the first type as  ``motivic'' theories, the second as ``geometric''. Assigning to a motivic theory $\sE^{*,*}$ its {\em geometric part} $X\mapsto \sE^{2*,*}(X)$ gives the link between these two notions.

Among the oriented motivic theories, the theory represented by Voevodsky's algebraic cobordism spectrum $\MGL$ is the universal one (see \cite{Voevodsky} for the construction of $\MGL$,  \cite{PaninPimenovRoendigs}  for its universality). For $S=\Spec k$, $k$ a field of characteristic zero, algebraic cobordism $\Omega^*$ as defined in \cite{LevineMorel}  is the universal geometric theory.  

For $S=\Spec k$ as above,   the universal property of $\Omega^*$ gives a canonical natural transformation $\Omega^*\to \MGL^{2*,*}$ of geometric theories. Relying on results of Hopkins-Morel-Hoyois \cite{HopkinsMorel, Hoyois},  this  was shown to be an isomorphism in \cite{LevineComparison}. The main purpose of this paper is to  extend this isomorphism property to other motivic theories and their geometric parts.

For a formal group law $(F, R)$, $F(u,v)\in R[[u,v]]$, one can form the oriented  geometric theory $X\mapsto R\otimes_\L\Omega^*(X)$, where $\L$ is the Lazard ring and $\L\to  R$ the classifying homomorphism for $F$. In the motivic setting, the situation is more delicate, however, just as in the classical case, if $(F,R)$ is a {\em Landweber exact} formal group law, there is a corresponding oriented weak commutative ring spectrum $\MGL(R)$ with $\MGL(R)^{*,*}(X)\cong R\otimes_\L\MGL^{*,*}(X)$ (see \cite{NSO}).

One can go a bit farther by considering the {\em effective cover} $f_0\sE\to \sE$ of a $T$-spectrum $\sE\in\SH(S)$. Here $f_0$ is the truncation functor with respect to Voevodsky's slice tower. $f_0\sE$ is an analog of the classical -1 connective cover of a spectrum, and inherits many properties from $\sE$. In particular, for $\sE$ an oriented weak commutative ring $T$-spectrum, the effective cover $f_0\sE$ inherits from $\sE$ a canonical structure of an oriented weak commutative ring $T$-spectrum; the coefficient ring $R_0:=(f_0\sE)^{2*,*}(S)$ is just the degree $\le0$ part of the coefficient ring $R:=\sE^{*,*}(S)$ of $\sE$ (at least for $S$ the spectrum of a characteristic zero field, see theorem~\ref{thm:LexactGLExact}). The coefficients of the group law $F_\sE$ associated to $\sE$ actually lie in $R_0$, but even if $(F_\sE, R)$ is Landweber exact, it is usually the case that $(F_\sE, R_0)$ is not.  

For  $\sE$ the oriented spectrum $\MGL(R)$ associated to a Landweber exact formal group law, with  effective cover $\sE_0$, our main result  is the following: For $k$ a characteristic zero field, $S=\Spec k$,  the canonical natural transformations
\[
\sE^{2*,*}(k)\otimes_{\L^*}\Omega^*\to \sE^{2*,*},\
\sE_0^{2*,*}(k)\otimes_{\L^*}\Omega^* \to (f_0\sE)^{2*,*}
\]
are isomorphisms of geometric oriented cohomology theories on $\Sm/k$. We actually prove a stronger result  (see corollary~\ref{cor:main}) concerning the oriented Borel-Moore homology theories on $\Sch/k$ defined by $\sE$ and $f_0\sE$. 

The case of the Landweber exact theory $\sE^{*,*}$ follows immediately  from the case $\sE=\MGL$, proved in \cite{LevineComparison}, so our efforts are directed at the effective cover $\sE_0$. The main idea for these results  already appears in our treatment (with S. Dai) of the case of algebraic $K$-theory and its effective cover \cite{DaiLevine}. We axiomatize the situation via the notion of a {\em geometrically Landweber exact} motivic oriented theory (see definition~\ref{def:GeomLandEx}) and show that for such a theory $\sE$, the canonical map 
\[
\sE^{2*,*}(k)\otimes_{\L^*}\Omega^*(X)\to \sE^{2*,*}(X)
\]
is an isomorphism for all $X\in \Sm/k$ (theorem~\ref{thm:GeomLandExact}). Our result on a Landweber exact theory and its effective cover then follow once we show that these are both geometrically Landweber exact (theorem~\ref{thm:LexactGLExact}).

The paper is organized as follows: We begin by recalling some of the basic notions concerning oriented (weak) commutative ring spectra in the motivic stable homotopy category in \S\ref{sec:Orient}, where we also recall the main results on the universality of $\MGL$. In \S\ref{sec:slice} we recall basic facts about the slice tower in the motivic stable homotopy category; we discuss as well some issues of convergence of the slice spectral sequence. In \S\ref{sec:Cover} we introduce the effective cover of an oriented weak ring $T$-spectrum and show that it too defines an oriented weak commutative ring spectrum. We discuss Landweber exact theories and introduce the notion of a geometrically Landweber exact theory.   We discuss oriented duality theories on the category of smooth pairs over $k$ in \S\ref{sec:OrientDual}. This theory provides the link between geometric theories  and motivic theories via the operation of taking the geometric part of a motivic oriented theory.  We describe the relation of  oriented duality theories  with algebraic cobordism in \S\ref{sec:AlgCobordOrient} and put everything together and prove our main results in the final section, \S\ref{sec:Comp}.

I would like to thank the referee for a number of helpful comments and suggestions. 
\\
\\
\noindent {\em Notation and conventions}. We let $\Spc$ and $\Spc_\bullet$ denote the categories of simplicial sets and pointed simplicial sets, respectively, with homotopy categories $\sH$ and $\sH_\bullet$. $\Spt$ the category of spectra (for the usual suspension operator $\Sigma:=(-)\wedge S^1$) and $\SH$ the stable homotopy category.

We denote by $\G_m$ the pointed $S$-scheme $(\A^1_S-0_S,1_S)$. $\P^n_*$ will denote  the pointed $S$-scheme $\P^n_S$, with base-point $[1,0,\ldots,0]$; we similarly  give  the limit $\P^\infty$ the base-point   $[1,0,\ldots]$.

We use a base-scheme $S$ that is separated, regular, noetherian and of finite Krull dimension; we will at a certain point specialize to the case $S=\Spec k$, $k$ a field of characteristic zero.
$\Sch/S$ will denote the category of quasi-projective schemes over $S$ and $\Sm/S$ the full subcategory of smooth, quasi-projective schemes over $S$.  $\Spc(S)$ and $\Spc_\bullet(S)$ are respectively the categories of presheaves on $\Sm/S$ with values in $\Spc$ and $\Spc_\bullet$.  $\Spt_{S^1}(S)$ will denote the category of $S^1$-spectra over $S$, this being the category of presheaves of spectra on $\Sm/S$. We let $\Spt_T(S)$ denote the category of $T$-spectra in $\Spc_\bullet(S)$, with $T:=\A^1/\A^1\setminus\{0\}$.  

The categories $\Spc(S)$,  $\Spc_\bullet(S)$, $\Spt_{S^1}(S)$ and $\Spt_T(S)$  all have {\em motivic} model structures (the original source for the unstable theory is  \cite{MorelVoev}, see also \cite{Motivic, GoerssJardine} for a compact description. For the stable theory, we refer the reader to \cite{Jardine2}), with homotopy categories denoted $\sH(S)$, $\sH_\bullet(S)$, $\SH_{S^1}(S)$ and $\SH(S)$, respectively. The categories  $\SH_{S^1}(S)$ and $\SH(S)$ are   triangulated tensor categories with translation given by $S^1$-suspension; the tensor structure is constructed using symmetric spectrum versions of $\Spt_{S^1}(S)$ and $\Spt_T(S)$, with appropriate model structures, also discussed in \cite{Jardine2}. The unit for the tensor structure in $\SH(S)$ is the motivic sphere spectrum $\mS_S:=\Sigma^\infty_T S_+$.

We have suspension functors $\Sigma^{a,b}:\SH(S)\to \SH(S)$ defined for $a\ge b\ge0$ by $\Sigma^{a,b}(\sE)=\sE\wedge S^{a-b}\wedge\G_m^{\wedge b}$, and extending to $a,b\in\Z$ using the invertibility of $\Sigma^{1,0}$ and $\Sigma^{1,1}$. We have as well the $T$-suspension $\Sigma_T(\sE):=\sE\wedge T$, canonically isomorphic to $\Sigma^{2,1}$. In addition, we have infinite suspension functors $\Sigma_T^\infty:\sH_\bullet(S)\to \SH(S)$, $\Sigma_s^\infty:\sH_\bullet(S)\to \SH_{S^1}(S)$ and $\Sigma_{\G_m}^\infty:\SH_{S^1}(S)\to \SH(S)$ with respective right adjoints $\Omega_T^\infty:\SH(S)\to \sH_\bullet(S)$, $\Omega_{S^1}^\infty:\SH_{S^1}(S)\to \sH_\bullet(S)$ and $\Omega_{\G_m}^\infty:\SH(S)\to \SH_{S^1}(S)$.

We let $\L$ denote the Lazard ring, that is, the coefficient ring of the universal rank one commutative formal group law $F_\L\in \L[[u,v]]$. We let $\L^*$ denote $\L$ with the grading determined by $\deg a_{ij}=1-i-j$ if $F_\L(u,v)=u+v+\sum_{i,j\ge1}a_{ij}u^iv^j$ and let $\L_*$ denote $\L$ with the opposite grading $\L_n:=\L^{-n}$.

\section{Oriented ring $T$-spectra}\label{sec:Orient} 
We recall that a morphism $f:\sE\to \sF$ in a compactly generated triangulated category $\sT$ is a {\em phantom map} if for each compact object $A$ in $\sT$, the induced map $f_*:\Hom_\sT(A,\sE)\to \Hom_\sT(A, \sF)$ is zero; it is enough to check on compact objects of the form $\sX[n]$, where $\sX$ is an element in a given set of compact generators and $n\in\Z$.   The subset of phantom maps $\Hom_\sT(\sE, \sF)_{\ph}\subset \Hom_\sT(\sE, \sF)$ is clearly a two-sided ideal, so we may form the category $\sT/\ph$ with the same objects as $\sT$ and morphisms $\Hom_{\sT/\ph}(\sE, \sF):=\Hom_\sT(\sE, \sF)/\Hom_\sT(\sE, \sF)_{\ph}$. 

We will be mainly interested in the case $\sT=\SH(S)$, in which the shifted suspension spectra $\Sigma_T^n\Sigma_T^\infty X_+$, $X\in \Sm/S$, $n\in\Z$, form a set of compact generators. Thus, two maps $f, g:\sE\to \sF$ are equal modulo phantom maps if and only if $f_*=g_*$ as maps of the associated bi-graded cohomology theories $f_*, g_*:\sE^{*,*}\to \sF^{*,*}$ on $\Sm/S$. 

\begin{definition} A {\em commutative ring $T$-spectrum} is a $T$-spectrum $\sE\in \Spt_T(S)$ together with maps $\mu:\sE\wedge \sE\to \sE$, $1:\mS_S\to \sE$ such that $(\sE, \mu, 1)$ is a commutative monoid in $\SH(S)$. A {\em weak commutative ring $T$-spectrum} is a $T$-spectrum $\sE\in \Spt_T(S)$ together with maps $\mu:\sE\wedge \sE\to \sE$, $1:\mS_S\to \sE$ such that $(\sE, \mu, 1)$ is a commutative monoid in $\SH(S)/\ph$. 

For $\sE, \sF$ (weak) commutative ring $T$-spectra, a morphism $f:\sE\to \sF$ in $\SH(S)$ is a monoid map (resp. weak monoid map) if $f$ is a map of monoid objects in $\SH(S)$ (resp. in $\SH(S)/\ph$).

\begin{definition}\label{Def:Orientation}
An {\em orientation} on a weak commutative ring $T$-spectrum $(\sE, \mu, 1)$ is an element $c\in \sE^{2,1}(\P^\infty_*)$ such that the restriction  $c_{|\P^1}\in \sE^{2,1}(\P^1_*)$ corresponds to $1\in \sE^{0,0}(S)$ under the suspension isomorphism $\sE^{0,0}(S)\cong \sE^{2,1}(T)\cong  \sE^{2,1}(\P^1_*)$.
\end{definition}

A pair $(\sE, c)$ consisting of a weak  commutative ring $T$-spectrum $(\sE, \mu, 1)$ and an orientation $c$ is an {\em   oriented weak commutative ring $T$-spectrum}. We say that a  weak  commutative ring $T$-spectrum $(\sE, \mu, 1)$ is {\em orientable} if there is an orientation $c$ on $\sE$. We sometimes omit the explicit mention of the orientation $c$.

A morphism $f:(\sE, c)\to (\sE', c')$ of oriented weak commutative ring $T$-spectra is a morphism $f:\sE\to \sE'$ in $\SH(S)$ such that $f$ is a weak monoid map and $f_*(c)=c'$.
\end{definition}

\begin{ex}\label{ex:OrientUniv} The algebraic cobordism spectrum $\MGL$ has been studied in \cite{PaninPimenovRoendigs}. $\MGL$ is the $T$-spectrum $(\MGL_0, \MGL_1,\ldots)$ with $\MGL_n$ the Thom space $Th(E_n)$, with $E_n\to B\GL_n$ the universal $n$-plane bundle. $\MGL_S$ is a commutative ring $T$-spectrum in $\SH(S)$ with an orientation $c_\MGL\in\MGL^{2,1}(\P_*^\infty)$ given by noting that the diagram
\[
\xymatrix{\P^\infty&E_1\ar[l]\ar[r]&Th(E_1)=\MGL_1}
\]
induces an isomorphism  $\P^\infty_*\cong \MGL_1$ in $\sH_\bullet(S)$ and thereby a morphism $c_\MGL:\Sigma^\infty_T\P^\infty_*\to \Sigma^{2,1}\MGL$ in $\SH(S)$.

We recall the following result of Panin-Pimenov-R\"ondigs:
\begin{thm}[\hbox{\cite[theorem 1.1]{PaninPimenovRoendigs}}] \label{thm:PPRUniversal} For $\sE$ a  commutative  ring $T$-spectrum in $\SH(S)$, sending a monoid morphism $\phi:\MGL\to \sE$ to $\phi(c_\MGL)$ gives a bijection of the set of monoid maps $\phi$ with the set of orientations $c_\sE\in \sE^{2,1}(\P_*^\infty)$. 
\end{thm}
Given an oriented commutative ring $T$-spectrum $(\sE,c)$ in $\SH(S)$, we let
\[
\phi_{\sE,c}:\MGL\to \sE
\]
denote the corresponding morphism of commutative ring $T$-spectra.  

In fact, this result extends directly to the setting of oriented weak commutative ring spectra, replacing ``monoid map'' with ``weak monoid map''. Indeed, the proof in \cite{PaninPimenovRoendigs} reduces to proving certain identities in $
\sE^{*,*}(\MGL(n))$ or $\sE^{*,*}(\MGL(n)\wedge \MGL(n))$. It is shown in  \cite{PaninPimenovRoendigs} that the canonical map
\[
\sE^{*,*}(\MGL(n))\to\lim_\leftarrow\sE^{*,*}(\text{Th}(\sT(n,m)))
\]
is an isomorphism, where $\sT(n,m)\to \Gr(n, n+m)$ is the universal bundle, and $\text{Th}(-)$ is the Thom space.  Similarly 
\[
\sE^{*,*}(\MGL(n)\wedge\MGL(n))\cong\lim_\leftarrow\sE^{*,*}(\text{Th}(\sT(n,m))\wedge\text{Th}(\sT(n,m)))
\]
Their proof relies on showing that the maps on $\sE$-cohomology in the inverse system are surjective, hence the proofs only require the knowledge of $\sE$-cohomology on objects in $\Sm/S$, and therefore the proofs work for oriented weak commutative ring spectra without change.

In case $S=\Spec k$, $k$ a field, Vezzosi's proof of the universality of $\MGL$   \cite[theorem 4.3]{Vezzosi} is also  based on obtaining  identities in $\sE^{**}(\MGL(n))$ or $\sE^{**}(\MGL(n)\wedge\MGL(n))$ and thus can  also be adapted to the setting of oriented weak commutative ring $T$-spectra.
\end{ex}

\begin{rem}\label{rem:Unit} Let $\sE\in\SH(S)$ be a weak commutative ring $T$-spectrum. Let $t_\sE\in \sE^{1,1}(\G_m)$ be the element corresponding to the unit $1\in \sE^{0,0}(S)$ under the suspension isomorphism. By functoriality, $t_\sE$ gives a map of pointed sets
\[
t_\sE(X):\sO_X^\times(X)\to \sE^{1,1}(X).
\]
If $\sE$ admits an orientation $c_\sE\in \sE^{2,1}(\P^\infty_*)$, then $t_\sE(X)$ is a group homomorphism.\footnote{Letting $\mS$ denote the sphere spectrum and writing $[a]:=t_\mS(a)$, this follows from the identity $[ab]=[a]+[b]+H[a][b]$ ($H:\mS\wedge\G_m\to \mS$ the stable Hopf map) and the fact that $H$ goes to zero in any oriented theory $\sE$. Both these facts are proven by Morel in \cite[\S 6]{MorelLec}.} Using the $\sE^{*,*}(S)$-module structure on $\sE^{*,*}(X)$, $t_\sE(X)$
extends to a map of $\sE^{*,*}(S)$-modules
\[
t_\sE(X):\sE^{2*,*}(S)\otimes_\Z\sO_X^\times(X)\to \sE^{2*+1,*+1}(X).
\]
\end{rem}

\section{The slice spectral sequence}\label{sec:slice}
Voevodsky introduced in \cite{VoevOpen} his {\em slice tower} as a motivic analog to the classical Postnikov tower for spectra. In addition to \cite{VoevOpen}, we refer the reader to \cite{VoevSlice1bis}  and \cite{LevineHptyComp, LevineHC} for the basic facts  concerning the slice tower, some of which we briefly recall here for the reader's convenience.

We consider the  localizing subcategory $\Sigma^p_T\SH^\eff(S)$ of $\SH(S)$ generated by objects $\Sigma^n_TX_+$, with $X\in \Sm/S$ and $n\ge p$, with inclusion $i_p:\Sigma^p_T\SH^\eff(S)\to \SH(S)$. The objects $\Sigma^n_TX_+$ are compact, hence each of the categories  $\Sigma^p_T\SH^\eff(S)$ is compactly generated;   the set of all  $\Sigma^n_TX_+$, $n\in \Z$, $X\in \Sm/S$ similarly forms a set of compact generators for $\SH(S)$. 

Thanks to results of Neeman \cite{NeemanTriang} on compactly generated triangulated categories, the functor $i_p$ admits a right adjoint  $r_p:\SH(S)\to \Sigma^p_T\SH^\eff(S)$. Defining $f_p:=i_p\circ r_p:\SH(S)\to \SH(S)$, one has canonical natural transformations $f_p\to \id$, $f_{p+1}\to f_p$. Applying these to a $T$-spectrum $\sE\in\SH(S)$ yields the  the slice tower
\[
\ldots\to f_{p+1}\sE\to f_p\sE\to\ldots\to \sE.
\]

The {\em slice functor} $s_p:\SH(S)\to \SH(S)$ is  characterized up to unique isomorphism by natural transformations $f_p\to s_p\to f_{p+1}[1]$ so that for each $\sE\in \SH(S)$, the triangle
\[
f_{p+1}\sE\to f_p\sE\to s_p\sE\to f_{p+1}\sE[1]
\]
is distinguished.

There are canonical isomorphisms for all $a, b, p$, 
\[
\Sigma^{a,b}\circ f_p\cong f_{p+b}\circ \Sigma^{a,b};\quad \Sigma^{a,b}\circ s_p\cong s_{p+b}\circ \Sigma^{a,b},
\]
compatible with the defining distinguished triangles for the $s_p$.  The canonical map $f_n\sE\to \sE$ induces an isomorphism $\Omega^\infty_{\G_m}f_n\sE\cong \Omega^\infty_{\G_m}\sE$ for all $n\le 0$.

For $Y\in \Sm/S$, we have the associated slice spectral sequence
\begin{equation}\label{eqn:SliceTower0}
E_2^{p,q}(Y;n)=(s_{-q}\sE)^{p+q, n}(Y)\Longrightarrow \sE^{p+q,n}(Y).
\end{equation}

We conclude this section with a convergence criterion for the spectral sequence  \eqref{eqn:SliceTower0}.

\begin{lem} \label{lem:SliceConv} Suppose that $S=\Spec k$, $k$ a perfect field. Take $\sE\in \SH(S)$. Suppose that there is a non-decreasing function $f:\Z\to \Z$ with $\lim_{n\to \infty}f(n)=\infty$, such that $\pi_{a+b,b}\sE=0$ for $a\le f(b)$. Then the for all $Y$, the spectral sequence \eqref{eqn:SliceTower0} is strongly convergent.\footnote{As spectral sequence $\{E_r^{pq}\}\Rightarrow G^{p+q}$ {\em converges strongly} to $G^*$ if for each $n$, the spectral sequence filtration $F^*G^n$ on $G^n$ is finite and exhaustive, there is an $r(n)$ such that  for all $p$ and all $r\ge r(n)$, all differentials entering and leaving $E_r^{p, n-p}$  are zero and the resulting maps $E_r^{p, n-p}\to E_\infty^{p, n-p}=\gr^p_FG^n$ are all isomorphisms.}  
\end{lem}

\begin{proof}  We may rewrite the $E_2$-term in \eqref{eqn:SliceTower0}  as
\[
E_2^{p,q}(Y;n)=\Hom_{\SH(S)}(\Sigma^\infty_TY_+, \Sigma^{p+q,n}s_{-q}\sE).
\]
The spectral sequence will be strongly convergent (for fixed $n$) if for each $m_0\in \Z$,  there is an integer $q(m_0)$ such that 
\[
\Hom_{\SH(S)}(\Sigma^\infty_TY_+, \Sigma^{m,n}f_q\sE)=0
\]
for all $q\ge q(m_0)$ and all  $m\ge -m_0$. Indeed, we have
\begin{align*}
\Hom_{\SH(S)}(\Sigma^\infty_TY_+, \Sigma^{m,n}f_q\sE)&=\Hom_{\SH_{S^1}(S)}(\Sigma^{-m}_{S^1}\Sigma^\infty_{S^1}Y_+, \Omega^\infty_{\G_m}\Sigma^{0,n}f_q\sE)\\
&=\pi_{-m}(\Omega^\infty_{\G_m}\Sigma^{0,n}f_q\sE(Y)).
\end{align*}
Thus, the above condition is just saying that, given an integer $m_0$,  the spectrum $\Omega^\infty_{\G_m}\Sigma^{0,n}f_q\sE(Y)$ is $m_0$-connected for all $q\ge q(m_0)$, this being a standard criterion for the strong convergence of the spectral sequence associated to the tower of spectra
\[
\ldots\to  \Omega^\infty_{\G_m}\Sigma^{0,n}f_q\sE(Y)\to\ldots\to 
\Omega^\infty_{\G_m}\Sigma^{0,n}f_{-n}\sE(Y)\cong \Omega^\infty_{\G_m}\Sigma^{0,n}\sE(Y).
\]
We proceed to exhibit the existence of such an integer $q(m_0)$.

We recall the  2-variable Postnikov tower $f_{a,b}$ defined in  \cite[\S 3]{LevineHptyComp}, with $f_b=f_{-\infty, b}$. 
We also recall that
\begin{enumerate}
\item[i)]  $\pi_{m+r,r}f_{a,b}\sF=0$ for $m<a$, $b\in \Z$  \cite[lemma 4.4]{LevineHptyComp},
\item[ii)]  the canonical map $\rho_{a,b}:f_{a,b}\sF\to f_b\sF$  induces an isomorphism on $\pi_{m+r,r}$ for all $m\ge a$, $r\ge b$,
\item[iii)]  $\rho_{a,b}:f_{a,b}\sF\to f_b\sF$ is an isomorphism if  $\rho$ induces an isomorphism on $\pi_{m+r,r}$ for all $m\in\Z$ and $r\ge b$ \cite[lemma 4.6]{LevineHptyComp}.
\end{enumerate}

The universal property of $f_{q+n}$ gives the isomorphism for all $r\ge q+n$
\[
\pi_{m+r,r}f_{q+n}\Sigma^{0,n}\sE\cong  \pi_{m+r,r}\Sigma^{0,n}\sE\cong \pi_{m+r,r-n}\sE.
\]
By assumption on $\sE$, $\pi_{m+n+r-n,r-n}\sE=0$ for $m+n\le f(r-n)$, and thus 
\[
\pi_{m+r,r}f_{q+n}(\Sigma^{0,n}\sE)=0
\]
for $m\le f(r-n)-n$ as well (assuming $r\ge q+n$). 

Now take $q(m_0)$ to be an integer such that   $f(q(m_0))\ge m_0+\dim Y+n$, and take $q\ge  q(m_0)$. As $f$ is non-decreasing, we have $f(r-n)-n\ge f(q(m_0))-n\ge m_0+\dim Y$ for all   $r\ge q+n$, hence 
\[
\pi_{m+r,r}f_{q+n}(\Sigma^{0,n}\sE)=0
\]
for $m\le m_0+\dim Y$, $r\ge q+n$. But by (i)-(iii),  this implies that the map
\[
f_{m_0+\dim Y+1, q+n}(\Sigma^{0,n}\sE)\to f_{q+n}(\Sigma^{0,n}\sE)
\]
is an isomorphism, and thus by (i), $\pi_{a+b,b} f_{q+n}(\Sigma^{0,n}\sE)=0$ for all $a\le m_0+\dim Y$, $b\in\Z$. Translating back to $\Omega^\infty_{\G_m}\Sigma^{0,n}f_q\sE$ via the isomorphisms
\[
\pi_{a+b,b} f_{q+n}(\Sigma^{0,n}\sE)\cong \pi_{a+b,b}\Sigma^{0,n}f_{q}\sE\cong \pi_{a+b,b} \Omega^\infty_{\G_m}\Sigma^{0,n}f_q\sE
\]
for $b\ge0$, $a\in\Z$, we have
\[
\pi_{a+b,b} \Omega^\infty_{\G_m}\Sigma^{0,n}f_q\sE=0
\]
for all $a\le m_0+\dim Y$, $b\ge 0$, $q\ge q(m_0)$. In particular, taking $b=0$, 
\[
\pi_a \Omega^\infty_{\G_m}\Sigma^{0,n}f_q\sE=0
\]
for all $a\le m_0+\dim Y$, $q\ge q(m_0)$.

We now apply the  local-global spectral sequence (with $q\ge q(m_0)$)
\[
E_2^{a,b}=H^a(Y_\Nis, \pi_{-b} \Omega^\infty_{\G_m}\Sigma^{0,n}f_q\sE)\Longrightarrow \pi_{-a-b}(\Omega^\infty_{\G_m}\Sigma^{0,n}f_q\sE(Y)).
\]
This sequence is convergent by \cite[theorem 8]{Brown} since $Y$ has finite Nisnevich cohomological dimension $\le \dim Y$ and is strongly convergent since $\pi_{-b}\Omega^\infty_{\G_m}\Sigma^{0,n}f_q\sE=0$ for $-b\le m_0+\dim Y$. In particular, $E_2^{a,b}=0$ for $a+b\ge-m_0$,   hence
$\Omega^\infty_{\G_m}\Sigma^{0,n}f_q\sE(Y)$ is $m_0$-connected for $q\ge q(m_0)$
\end{proof}

\section{Oriented theories and their $T$-effective covers}\label{sec:Cover}
\subsection{The $T$-effective cover}
\begin{prop} 1. Let $(\sE, \mu, 1)$ be a weak commutative ring $T$-spectrum. Then $f_0\sE$ has a unique structure of a weak commutative ring $T$-spectrum such that the canonical map $\rho:f_0\sE\to \sE$ is a weak monoid  morphism.\\
2. If $(\sE, c)$ is an oriented weak commutative  ring $T$-spectrum, then there is a unique element $c_0\in (f_0\sE)^{2,1}(\P^\infty_*)$ with $\rho_*(c_0)=c$, and $c_0$ defines an orientation on the weak commutative ring $T$-spectrum $f_0\sE$.\\
3. We have $t_\sE=\rho_*(t_{f_0\sE})$ in $\sE^{1,1}(\G_m)$.
\end{prop}

\begin{proof} Write $\sE_0$ for $f_0\sE$. We first lift the multiplication. We note that $\SH^\eff(S)$ is closed under smash product, as the generators are, so $\sE_0\wedge \sE_0$ is in $\SH^\eff(S)$. Thus, the composition 
\[
\sE_0\wedge \sE_0\xrightarrow{\rho\wedge \rho} \sE\wedge \sE\xrightarrow{\mu}\sE
\]
admits a unique lifting $\mu_0:\sE_0\wedge\sE_0\to \sE_0$ making
\[
\xymatrix{
\sE_0\wedge \sE_0\ar[r]^-{\mu_0}\ar[d]_{\rho\wedge\rho}&\sE_0\ar[d]^\rho\\
\sE\wedge \sE\ar[r]_-\mu&\sE
}
\]
commute. 

We claim that $\mu_0$ is associative modulo phantom maps. For this, let 
\[
\phi:\sE_0\wedge\sE_0\wedge \sE_0\to \sE_0
\]
denote the difference $\mu_0\circ(\mu_0\wedge\id)-\mu_0\circ(\id\wedge\mu_0)$. Let $A$ be a compact object of $\SH(S)$ and let $h:A\to \sE_0\wedge\sE_0\wedge \sE_0$ be a morphism in $\SH(S)$.  We apply \cite[theorem 4.3.3]{NeemanTriang}    for the  subcategory $\SH^\eff(S)$ of $\SH(S)$, generated as a localizing subcategory of $\SH(S)$ by the set of compact objects $\sS:=\{\Sigma_T^\infty X_+, X\in \Sm/S\}$, with cardinal $\beta=\aleph_0$. This tells us that there is a compact object $B$ in $\SH^\eff(S)$, and morphisms $f:A\to B$, $g:B\to \sE_0\wedge\sE_0\wedge \sE_0$, such that $h=g\circ f$.  But as $B$ is compact, the composition $\rho\circ\phi\circ g$ in $\SH(S)$ is zero, and since $B$ is in $\SH^\eff(S)$,   the universal property of $\rho:\sE_0\to \sE$ implies that $\phi\circ g=0$ in $\SH^\eff(S)$. Thus $\phi\circ h=\phi\circ g\circ f=0$ and hence  $\mu_0$ is associative modulo phantom maps.

The unit map $1:\mS_S\to \sE$ factors uniquely through $\rho$, since $\mS_S$ is in $\SH^\eff(S)$, giving the unit map $1_0:\mS_S\to \sE_0$; the identities 
\[
\mu_0\circ(\id_{\sE_0}\wedge 1_0)=\id_{\sE_0}=\mu_0\circ(1_0\wedge \id_{\sE_0})
\]
in $\SH(S)/\ph$ and the commutativity of $\mu_0$ modulo phantom maps follow as for the proof of associativity, as does the uniqueness of $\mu_0$  modulo phantom maps.

For (2), recall that $*=[1,0,\ldots, 0]$ in $\P^n$.  Let $\A^n(*)\subset \P^n$ be the  open subscheme $X_0\neq0$. Then the quotient map $\P^n_*\to \P^n/\A^n(*)$ is an isomorphism in $\sH_\bullet(S)$, giving the isomorphism $\P^\infty_*\cong \colim_n\P^n/\A^n(*)$. In particular, both $\P^n_*$ and $\P^\infty_*$ are in $\Sigma^1_T\SH^\eff(S)$.

We have the isomorphism $\Sigma_Tf_0\sE\cong f_1\Sigma_T\sE$, and under this isomorphism $\Sigma_T\rho:\Sigma_Tf_0\sE\to \Sigma_T\sE$ goes over to the universal map $\rho_1:f_1\Sigma_T\sE\to \Sigma_T\sE$. Thus $c:\P^\infty_*\to \Sigma_T\sE$ factors uniquely through $\Sigma_T\rho$ via a map $c_0:\P^\infty_*\to \Sigma_Tf_0\sE$. The restriction of $c_0$ to $\P^1_*$ is similarly the unique lifting of $c_{|\P^1_*}$, which shows that $c_{0|\P^1_*}\in (f_0\sE)^{2,1}(\P^1_*)$ corresponds to $1_0\in (f_0\sE)^{0,0}(S)$ under the suspension isomorphism. Thus $c_0\in (f_0\sE)^{2,1}(\P^\infty_*)$ is an orientation for $f_0\sE$.

For (3), this follows directly from (1) and the naturality of the map $t_\sE$ with respect to weak monoid morphisms.
\end{proof}

\subsection{Landweber exact theories} We recall the Landweber exactness conditions for a formal group law $(F, R)$: Let $\phi:\L^*\to R$ be the classifying map for $(F,R)$. Choose homogeneous polynomial generators for  $\L^*$ over $\Z$, $\L^*=\Z[x_1, x_2,\ldots]$, with $\Deg\, x_i=-i$. Then for each prime $p$, the sequence $p, x_{p-1}, \ldots, x_{p^i-1},\ldots$ is a regular sequence on $R$.

For an oriented theory  $(\sE,c)$, $\sE\in \SH(S)$, let $R_\sE^*$ be the {\em coefficient ring}, that is, $R_\sE^*:=\sE^{2*,*}(S)$. Let $F_{\sE,c}(u,v)\in R_\sE^*[[u,v]]$ be the formal group law of the oriented theory $(\sE,c)$.

We let $\SH(S)_\sT\subset \SH(S)$ be the full subcategory of Tate spectra, that is, the localizing subcategory generated by the spheres $\Sigma^{a,b}\mS_S$, $a, b\in\Z$.  We recall  from \cite[\S4]{NSO} some basic facts concerning Tate spectra: There is an exact projection functor $p_\sT:\SH(S)\to \SH(S)_\sT$ right adjoint to the inclusion. For a morphism $S\to S'$, the functor $Lf^*:\SH(S')\to \SH(S)$ maps $\SH(S')_\sT$ to $\SH(S)_\sT$ and the restriction $Lf^*_\sT:\SH(S')_\sT\to \SH(S)$ has right adjoint $p_\sT\circ Rf_*$. Finally, $p_\sT$ is a $\SH(S)_\sT$-module map, that is, for $\sE\in \SH(S)_\sT$ and $\sF\in \SH(S)$, there is a natural isomorphism
\begin{equation}\label{eqn:TateMod}
p_\sT(\sE\wedge \sF)\cong \sE\wedge p_\sT(\sF).
\end{equation}
We let $\SH(S)_\fin\subset \SH(S)$ denote the thick subcategory of {\em finite spectra}, namely, the thick subcategory generated by the spectra $\Sigma^n_T\Sigma_T^\infty X_+$, $X\in \Sm/S$, $n\in \Z$.  Similarly, we let $\SH(S)_{\sT,\fin}\subset \SH(S)_\sT$ be the thick subcategory generated by spheres $\Sigma^{a,b}\mS_S$, $a, b\in\Z$.

\begin{definition}\label{def:LWE} Let $S=\Spec k$, $k$ a field of characteristic zero. An oriented  weak commutative ring $T$-spectrum $(\sE, c)$ with classifying map $\phi_{\sE, c}:\MGL\to \sE$  is said to be {\em Landweber exact} if \\
1. The classifying map $\L^*\to R_\sE^*$ for the formal group law $F_{\sE,c}$ satisfies the Landweber exactness conditions.\\
2. For all finite spectra $\sF\in\SH(k)_\fin$, the map $R_\sE^*\otimes_{R^*_\MGL}\MGL^{*,*}(\sF)\to \sE^{*,*}(\sF)$ induced by $\phi_{\sE,c}$ and the product map $\sE^{2n,n}(k)\otimes \sE^{a,b}(\sF)\to \sE^{a+2n,b+n}(\sF)$ is an isomorphism.\\
3. $\sE$ is in the Tate subcategory $\SH(k)_\sT$ of $\SH(k)$.
\end{definition}
 
\begin{rem} Let $k$ be a field of characteristic zero. Naumann-Spitzweck-{\O}stv{\ae}r \cite[theorem 8.7]{NSO} show that, for each Landweber exact $\L^*$-algebra $R^*$, the bigraded functor from finite spectra to bi-graded algebras, $\sF\mapsto R^*\otimes_{\L^*}\MGL^{*,*}(\sF)$, is represented by an object $\MGL(R^*)$ in $\SH(k)_\sT$ with morphisms $\mu:\MGL(R^*)\wedge\MGL(R^*)\to \MGL(R^*)$, $1:\mS_k\to \MGL(R^*)$ defining a (oriented) weak commutative ring $T$-spectrum.  It does not seem to be known if one can give  $\MGL(R^*)$ the structure of an oriented commutative ring $T$-spectrum. 

In fact, Naumann-Spitzweck-{\O}stv{\ae}r work in the setting of homology theories on $\SH(S)$ rather than cohomology theories on finite spectra, and they prove their results for $S$ a regular, noetherian separated scheme of finite Krull dimension.  In case the base-scheme is $\Spec k$, $k$ a field of characteristic zero, all finite spectra in $\SH(k)$ are strongly dualizable \cite[theorem 1.4]{Riou}, so one may easily pass from homology theories on $\SH(k)$ to cohomology theories on finite spectra in $\SH(k)$.
 \end{rem}

In the case of a general base-scheme $S$, the classifying map $\L^*\to \MGL^{2*,*}(S)$ might not be an isomorphism, and one would need to modify the definition of a Landweber exact spectrum  accordingly. Naumann-Spitzweck-{\O}stv{\ae}r simply define a Landweber exact spectrum to be one of the form $\MGL(R^*)$, rather than giving an intrinsic characterization.
In case $S=\Spec k$, $k$ a field of characteristic zero, the construction of Naumann-Spitzweck-{\O}stv{\ae}r in fact gives us all Landweber exact oriented weak commutative ring spectra, as the next two lemmas show. 

\begin{lem}\label{lem:HomologyIso} Suppose $S=\Spec k$, $k$ a field of characteristic zero. Suppose that $(\sE, c)$  is Landweber exact (in the sense of definition~\ref{def:LWE}).  Then the classifying map $\phi_{\sE,c}:\MGL\to \sE$ and the product maps for $\sF\in \SH(k)$, 
$R_*^\sE\otimes_\Z\sE_{*,*}(\sF)\to \sE_{*,*}(\sF)$, together induce an isomorphism of homology theories on $\SH(k)$
\[
R_*^\sE\otimes_{R_*^\MGL}\MGL_{*,*}(-)\to \sE_{*,*}(-).
\]
\end{lem}

\begin{proof} Both functors $R_*^\sE\otimes_{R_*^\MGL}\MGL_{*,*}(-)$ and
$ \sE_{*,*}(-)$ are homological functors on $\SH(k)$, compatible with coproducts, hence the full subcategory $\sC$  of objects $\sF$ for which the lemma holds is a localizing subcategory of $\SH(k)$. As $\SH(k)$ is generated as a localizing category by $\SH(k)_\fin$, this reduces us to the case $\sF\in \SH(k)_\fin$.

Take $\sF$ in $\SH(k)_\fin$. Then $\sF$ is strongly dualizable with $\sF^D\in \SH(k)_\fin$  \cite[theorem 1.4]{Riou}, and we have $\sG_{*,*}(\sF)\cong \sG^{-*,-*}(\sF^D)$ for all $\sG\in \SH(k)$. The fact that $\phi_{\sE,c}$ induces an isomorphism
\[
R^*_\sE\otimes_{R^*_\MGL}\MGL^{*,*}(\sF^D)\cong  \sE^{*,*}(\sF^D).
\]
shows that  $\SH(k)_\fin\subset \sC$. 
\end{proof}

\begin{lem}\label{lem:LandweberIso} Suppose $S=\Spec k$, $k$ a field of characteristic zero.  Let $(\sE,c)$ be a Landweber exact oriented weak commutative ring $T$-spectrum  (in the sense of definition~\ref{def:LWE}). Then there is an isomorphism of oriented weak commutative ring $T$-spectra $\psi:\MGL(R_\sE^*)\to \sE$.
\end{lem}

\begin{proof}  Since $k$ has characteristic zero, the results of Hopkins-Morel-Hoyois \cite{HopkinsMorel, Hoyois} show that the classifying map $\L_*\to R_*^\MGL$ is an isomorphism. We may therefore replace $R_*^\MGL$ with $\L_*$ in our definition of Landweber exactness and in lemma~\ref{lem:HomologyIso}.

 Let $f:\Spec k\to \Spec \Q$ be the canonical morphism of schemes, giving the adjoint functors $Lf^*:\SH(\Q)\xymatrix{\ar@<3pt>[r]&\ar@<3pt>[l]}\SH(k):Rf_*$. Consider the object $Rf_*\sE\in \SH(\Z)$ and take $\sF\in \SH(\Z)_\sT$. We have a canonical isomorphism $Rf_*(\sE\wedge Lf^*\sF)\cong Rf_*\sE\wedge \sF$  (see e.g. \cite[pg. 578, eqn.~(26)]{NSO}) and therefore by lemma~\ref{lem:HomologyIso} we have 
\[
(Rf_*\sE)_{*,*}(\sF)=\sE_{*,*}(Lf^*\sF)\cong R^\sE_*\otimes_{\L_*}\MGL_{k *,*}(Lf^*\sF)
\]
for all $\sF$ in $\SH(\Q)_\sT$.  That is, $Rf_*\sE$ represents the homology theory $\sF\mapsto R^\sE_*\otimes_{\L_*}\MGL_{k *,*}(Lf^*\sF)$ on $\SH(\Q)_\sT$; using the isomorphism \eqref{eqn:TateMod}, the same holds for the Tate projection $p_\sT Rf_*\sE$. The spectrum $\MGL_\Q(R^*_\sE)\in \SH(\Q)$ similarly represents the homology theory $\sF\mapsto R^\sE_*\otimes_{\L_*}\MGL_{\Q *,*}(\sF)$. 

We have a canonical isomorphism $\MGL_k\cong Lf^*\MGL_\Q$; we consider this as an identity. The spectrum $\MGL_k(R_\sE^*)$ is similary by construction canonically isomorphic to the pull-back $Lf^*\MGL_\Q(R_\sE^*)$. Applying $Rf_*$ and the Tate projection to pair of canonical maps $\phi_{\MGL(R_\sE^*)}:\MGL_k\to \MGL_k(R_\sE^*)$, $\phi_\sE:\MGL_k\to\sE$
 and using the appropriate adjoint properties gives us the commutative diagram in $\SH(\Q)$
\begin{equation}\label{eqn:Diag0}
\xymatrix{
\MGL_\Q\ar[d]_{\phi_{\MGL_\Q(R_\sE^*)}}\ar[r]&p_\sT Rf_*\MGL_k\ar[d]_{p_\sT Rf_*\phi_{\MGL(R_\sE^*)}}\ar[r]\ar[dr]^{\hskip 10pt p_\sT Rf_*\phi_\sE}&Rf_*\MGL_k\ar[dr]^{Rf_*\phi_\sE}\\
\MGL_\Q(R_\sE^*)\ar[r]&p_\sT Rf_*\MGL_k(R_\sE^*) &p_\sT Rf_*\sE\ar[r]&Rf_*\sE
}
\end{equation}
This induces the commutative diagram of homology theories on $\SH(\Q)_\sT$
\[
\xymatrix{
\MGL_{\Q,**}(-)\ar[d]_{\phi_{\MGL_\Q(R_\sE^*)}}\ar[r]&\MGL_{k,**}(Lf^*(-))\ar[d]_{p_\sT Rf_*\phi_{\MGL_k(R_\sE^*)}} 
\ar[dr]^{\hskip 10pt p_\sT Rf_*\phi_\sE}\\
\MGL_\Q(R_\sE^*)_{**}(-)\ar[r]&\MGL_k(R_\sE^*)_{**}(Lf^*(-))\ar[r]_-{\alpha_{**}} &\sE_{**}(Lf^*(-)) 
}
\]
Here  $\alpha_{**}$ is the isomorphism given by lemma~\ref{lem:HomologyIso} and we use the identities 
\[
\pi_{**}(p_\sT(Rf_*\sF)\wedge \sG)(\Q)=\pi_{**}(Rf_*\sF\wedge \sG)(\Q)=\pi_{**}(\sF\wedge Lf^*\sG)(k)
\]
for $\sF\in \SH(k)$, $\sG\in \SH(\Q)_\sT$.

As $\SH(\Q)_\sT$ satisfies Brown representability, in the sense of \cite[definition 3.1]{Neeman}  (see \cite[theorem 1]{NaumannSpitzweck} and  \cite[proposition 4.11 and theorem 5.1]{Neeman}), the isomorphism $\alpha_{**}$ arises from an isomorphism $\alpha:p_\sT Rf_*\MGL_k(R_\sE^*)\to p_\sT Rf_*\sE$ that fills in the diagram \eqref{eqn:Diag0} to a commutative (up to phantoms) diagram. This gives us the diagram in $\SH(\Q)$
\[
\xymatrix{
\MGL_\Q\ar[d]_{\phi_{\MGL_\Q(R_\sE^*)}}\ar[r]^-\theta&Rf_*\MGL_k\ar[d]^{Rf_*\phi_\sE}\\
\MGL_\Q(R_\sE^*)\ar[r]_-{\tilde\psi}&Rf_*\sE,
}
\]
commutative after composition with a map $\sF\to \MGL_\Q$, $\sF\in \SH(\Q)_{\sT, \fin}$,
where $\theta$ is the unit of the adjunction and $\tilde\psi$ is the evident composition. Applying adjunction gives the diagram in $\SH(k)$ 
\[
\xymatrix{
\MGL_k\ar[d]_{\phi_{\MGL(R_\sE^*)}}\ar@{=}[r]&\MGL_k\ar[d]^{\phi_\sE}\\
\MGL_k(R_\sE^*)\ar[r]_-\psi&\sE
}
\]
which we claim commutes up to a phantom map. Indeed, $\MGL_\Q$ is a filtered colimit of finite Tate spectra $F_0\MGL_\Q\to F_1\MGL_\Q\to\ldots\to \MGL_\Q$, with each of the maps $F_n\MGL_\Q\to F_{n+1}\MGL_\Q$ a cofibration of cofibrant objects in $\Spt_T(\Q)_{\sT,\fin}$.  This writes $\MGL_k=Lf^*\MGL_\Q$ as the  filtered colimit of finite Tate  spectra $F_0\MGL_k\to F_1\MGL_k\to\ldots\to \MGL_k$, by taking $F_n\MGL_k\cong Lf^*F_n\MGL_\Q$. This gives us an isomorphism in $\SH(k)$
\[
\MGL_k\cong \hocolim_nF_n\MGL_k.
\]
 Let $i_n^\Q: F_n\MGL_\Q\to \MGL_\Q$ be the canonical map and let $i_n:F_n\MGL_k\to \MGL_k$ be the induced map.  As $F_n\MGL_\Q$ is in $\SH(\Q)_{\sT, \fin}$, we have 
$\tilde\psi\circ\phi_{\MGL_\Q(R_\sE^*)}\circ i_n^\Q=Rf_*\phi_\sE\circ i_n^\Q$.
Using adjointness again, we see that  $\psi\circ\phi_{\MGL_k(R_\sE^*)}\circ  i_n=\phi_\sE\circ i_n$ for all $n$. Thus, for $\sF\in\SH(k)$ a compact object and $g:\sF\to \MGL_k$ a morphism, $g$ factors through some $i_n$ and hence $\psi\circ\phi_{\MGL_k(R_\sE^*)}\circ  g=\phi_\sE\circ g$, which verifes our claim.  

As the maps 
\begin{gather*}
\cup\circ (id\otimes \phi_{\MGL_k(R_\sE^*)}):R_\sE^*\otimes_{\L^*}\MGL_k^{**}(X)\to \MGL_k^{**}(R_\sE^*)(X)\\
\cup\circ (id\otimes \phi_\sE):R_\sE^*\otimes_{\L^*}\MGL_k^{**}(X)\to \sE^{**}(X)
\end{gather*}
are isomorphisms for all $X\in\Sm/k$, the map $\psi$ induces an isomorphism of cohomology theories on $\SH(k)_\fin$. As above, this shows that  $\psi$ is an isomorphism of oriented weak commutative ring spectra in  $\SH(k)$.
\end{proof}

\subsection{Geometrically Landweber exact theories}\label{subsec:LandweberTheories} 
We specialize to the setting $S=\Spec k$, $k$ a field of characteristic zero. As above, let $(\sE, c)$ be a oriented weak commutative ring $T$-spectrum in $\SH(k)$.

As we have noted in remark~\ref{rem:Unit}, we have the element $t_\sE\in \sE^{1,1}(\G_m)$ corresponding to the unit $1\in \sE^{0,0}(k)$ under the suspension isomorphism, inducing the group homomorphism
\[
t^1_\sE(X):\sO_X^\times(X)\to \sE^{1,1}(X).
\]
In addition, the $\sE^{*,*}(k)$-module structure on $\sE^{*,*}(X)$ 
extends  $t^1_\sE(X)$ to a map of $\sE^{*,*}(k)$-modules
\[
t_\sE(X):\sE^{2*,*}(k)\otimes_\Z\sO_X^\times(X)\to \sE^{2*+1,*+1}(X).
\]
Furthermore, the classifying map $\phi_{\sE,c}:\MGL\to \sE$ combined with the product in $\sE$-cohomology gives rise to the 
homomorphism 
\[
\phi_{\sE, X}:R^*_\sE\otimes_{R^*_\MGL}\MGL^{2*-\epsilon, *}(X)\to \sE^{2*-\epsilon,*}(X)
\]
for  $X\in \Sm/k$, natural with respect to morphisms in $\Sm/k$. If $\eta$ is a generic point of $X$, we may pass to the limit over Zariski open neighborhoods of $\eta$, giving the homomorphism 
\[
\phi_{\sE, \eta}:R^*_\sE\otimes_{R^*_\MGL}\MGL^{2*-\epsilon, *}(\eta)\to \sE^{2*-\epsilon,*}(\eta).
\]

\begin{rem}
By the Hopkins-Morel-Hoyois theorem \cite{Hoyois}, the map $\rho_\MGL:\L^*\to \MGL^{2*,*}(k)$ classifying the formal group law $F_\MGL$ is an isomorphism;  in particular, $\MGL^{2n,n}(k)=0$ for  $n>0$. We henceforth identify $\L^*$ and $\MGL^{2*,*}(k)=R^*_\MGL$ via $\rho_\MGL$, and use $\L^*$ and $R^*_\MGL$ interchangeably. 
\end{rem}

\begin{definition} \label{def:GeomLandEx} We say that  an oriented weak commutative ring $T$-spectrum $(\sE,c)$ in $\SH(k)$ is {\em geometrically Landweber exact} if for all generic points $\eta\in X\in\Sm/k$, the map
\[
\phi_{\sE, \eta}:R^*_\sE\otimes_{R^*_\MGL}\MGL^{2*-\epsilon, *}(\eta)\to \sE^{2*-\epsilon,*}(\eta)
\]
is an isomorphism for $\epsilon=0$ and a surjection for $\epsilon=1$.
\end{definition}
As we are assuming the $k$ has characteristic zero, and hence is perfect, this definition would be equivalent to one requiring the above maps to be isomorphisms for all points $\eta\in Y\in\Sch/k$, as the closure of $\eta$ in $Y$ will contain a dense open subscheme $X$, smooth and quasi-projective over $k$.

\begin{prop} \label{prop:MGLLWE}  
1. For  $\eta\in X\in \Sm/k$, let $p_\eta:\eta\to \Spec k$ be the structure morphism. Then $p_\eta^*: \MGL^{2*,*}(k)\to  \MGL^{2*,*}(\eta)$ is an isomorphism\\
2.  For $\eta\in X\in \Sm/k$,  the map $t_\MGL(\eta):\MGL^{2*,*}(k)\otimes_\Z k(\eta)^\times\to \MGL^{2*+1,*+1}(\eta)$ is an isomorphism.\\
3. For $Y\in\Sm/k$, $n\in\Z$, $\MGL^{2n+a,n}(Y)=0$ for all $a>0$.
\end{prop}

\begin{proof} We use the (strongly convergent) Hopkins-Morel spectral sequence \cite{Hoyois}
\begin{equation}\label{eqn:HMSS}
E_2^{p,q}(n):=\L^{q}\otimes H^{p-q}(Y,\Z(n-q))\Longrightarrow \MGL^{p+q,n}(Y).
\end{equation}
Recall that for $Y\in \Sm/k$, $H^a(Y,\Z(b))=0$ for  $a>2b$. Thus $E_2^{p,q}(n)=0$ for $p+q>2n$,   hence $\MGL^{m,n}(Y)=0$ for  $m>2n$, $Y\in\Sm/k$, proving (3). 

For $\eta\in X\in \Sm/k$, we have   $H^a(\eta,\Z(b))=0$ for $a>b\ge0$  or $b=0$ and $a\neq0$ or $b<0$. Thus
\[
E_2^{2n-q, q}(n)=\L^q\otimes H^{2(n-q)}(k,\Z(n-q))=0
\]
if $q\neq n$. Similarly,   
\[
E_2^{2n-q-1, q}(n)=\L^q\otimes H^{2(n-q)-1}(k,\Z(n-q))=0
\]
if $q\neq n-1$. Since $H^0(\eta,\Z(0))=\Z$,  $H^1(\eta,\Z(1))=k(\eta)^\times$, the surviving non-zero terms contributing to $\MGL^{2n,n}(\eta)$ and $\MGL^{2n-1,n}(\eta)$ are 
\[
E_2^{n,n}(n)=\L^n;\ n\le0,\quad E_2^{n,n-1}(n)=\L^{n-1}\otimes k(\eta)^\times;\ n\le 1.
\]
Similarly,  the terms 
\begin{gather*}
E_2^{n+r, n-r+1}=\L^{n-r+1}\otimes H^{2r-1}(\eta, \Z(r-1)),\
E_2^{n+r, n-r}=\L^{n-r}\otimes H^{2r}(\eta,\Z(r)),\\  E_2^{n-r, n+r-2}=\L^{n+r-2}\otimes H^{-2r+2}(\eta,\Z(2-r)),\\
E_2^{n-r, n+r-1}=\L^{n+r-1}\otimes H^{-2r+1}(\eta,\Z(1-r))
\end{gather*}
are all zero for $r\ge2$. The part of the spectral sequence computing $\MGL^{2n,n}(\eta)$  and  $\MGL^{2n-1,n}(\eta)$ thus degenerates at $E_2$,  giving  isomorphisms $\L^n\cong \MGL^{2n,n}(\eta)$ and   $\L^{n-1}\otimes k(\eta)^\times\cong \MGL^{2n-1,n}(\eta)$ for all $n$.  

As the spectral sequence is natural in $Y$ and the the pullback $p_\eta^*:H^0(k, \Z(0))\to H^0(\eta,\Z(0))$ is an isomorphism,    $p_\eta^*: \MGL^{2*,*}(k)\to  \MGL^{2*,*}(\eta)$ is an isomorphism, proving (2).

To complete the proof, we must check that the isomorphisms $\L^q\to \MGL^{2q,q}(k)$ given above arise from the classifying map $\L^*\to \MGL^{2*,*}(k)$,  and also that  the  isomorphisms $\L^q\otimes k(\eta)^\times\to \MGL^{2q+1, q+1}(\eta)$ are induced by $t^1_\MGL(\eta)$ and the isomorphism $\L^q\to \MGL^{2q,q}(k)$.

The fact that the isomorphisms  $\L^q\to \MGL^{2q,q}(k)$ are the maps coming from the homomorphism $\L^*\to \MGL^{2*,*}(k)$ classifying the formal group law for $\MGL^{*,*}$  follows directly from construction of the spectral sequence in  \cite{Hoyois}:  the spectral sequence arises from a choice of polynomial generators $x_1, x_2, \ldots$ for $\L^*$, which are then considered as maps $x_n:\mS_k\to \Sigma^{2n,n}\MGL$ via the classifying map, and these maps are in turn used to construct the tower which gives rise to the spectral sequence (which is then identified with the slice tower for $\MGL$).

Let us now check that the map $t^1_\MGL(\eta):k(\eta)^\times \to \MGL^{1,1}(\eta)$ is the isomorphism given by the spectral sequence. Since $f_0\MGL=\MGL$, we have the   distinguished triangle $f_1\MGL\to \MGL\xrightarrow{p_0} s_0\MGL\to f_1\MGL[1]$. Via the isomorphism $s_0\MGL\cong M\Z$, the unit for $\MGL$ goes to the unit in $M\Z$, and thus $\MGL\to M\Z$ induces a commutative diagram
\[
\xymatrix{
\G_m\ar[r]^-{t_\MGL}\ar[dr]_{t_{M\Z}}&\Sigma_{\G_m}\MGL\ar[d]^{\Sigma_{\G_m}p_0}\\&\Sigma_{\G_m}M\Z}. 
\]
But the spectral sequence $E^{**}_*(1)$ arises by taking the slice tower for $\MGL$ and applying $\Sigma_{\G_m}$. The spectral sequence computation we have just made shows that the isomorphism $\MGL^{1,1}(\eta)\to H^1(\eta,\Z(1))$ arises from the edge homomorphism associated to the map
$\Sigma_{\G_m}p_0$, which is split by composing the inverse of the isomorphism $t^1_{M\Z}:k(\eta)^\times\to H^1(\eta,\Z(1))$ with $t^1_\MGL$. This verifies our assertion for $t^1_\MGL(\eta)$. 

The assertion (2) follows from this and the product structure on the slice spectral sequence (see \cite{Pelaez}).
\end{proof}

Our next result shows that passing from a Landweber exact theory $(\sE, c)$ to its effective cover $(f_0\sE, c_0)$ results in a geometrically Landweber exact theory with coefficient ring the evident truncation of the coefficient ring $R_\sE^*$. In fact, it follows directly from the universal property of the slice truncation that $R_{f_0\sE}^n=R_\sE^n$ for $n\le0$; the vanishing of $R_{f_0\sE}^n$ for $n>0$ is less evident and our proof requires the Hopkins-Morel-Hoyois spectral sequence.

\begin{thm} \label{thm:LexactGLExact} Let $(\sE, c)$ be an oriented weak commutative ring $T$-spectrum. Suppose that $(\sE,c)$ is Landweber exact. Then\\
1. Both $(\sE,c)$ and the effective cover $(f_0\sE, c_0)$ are geometrically Landweber exact. \\
2. The coefficient ring of $f_0\sE$ is the graded subring of $R_\sE^*$ given by
truncation: 
\[
R^n_{f_0\sE}=\begin{cases} R^n_\sE&\text{ if } n\le 0\\ 0&\text{ if } n>0.\end{cases}
\]
3. The slice spectral sequences for $\sE^{*,*}(Y)$ and for $f_0\sE^{*,*}(Y)$ are strongly convergent for all $Y\in\Sm/k$.
\end{thm}

\begin{proof} We first show (3). As $f_q(f_0\sE)=f_q(\sE)$ for $q\ge0$, it suffices to show that  slice spectral sequences for $\sE^{*,*}(Y)$ is strongly convergent for all $Y\in \Sm/k$. 

We have already seen that $\MGL^{2n+a,n}(Y)=0$ for all $Y\in\Sm/k$, $a>0$, $n\in\Z$. Since $\sE$ is Landweber exact, this shows that $\sE^{2n+a,n}(Y)=0$ for all  $Y\in\Sm/k$, $a>0$, $n\in\Z$, and thus $\pi_{m+r,r}(\sE)=0$ for all $m\le r-1$, $r\in\Z$. We may therefore apply lemma~\ref{lem:SliceConv} with $f$ the function $f(r)=r-1$, giving the convergence.

We now prove (1) and (2) by comparing the slice spectral sequences for $\sE$ and $f_0\sE$. For  $\sF$ a $T$-spectrum and $Y\in\Sm/k$,  the $E_2$-term is given by
\[
E_2^{p,q}(n)(\sF,Y)=[\Sigma^\infty_TY_+, \Sigma^{p+q,n}s_{-q}\sF]_{\SH(k)}.
\]
Since $s_{-q}f_0\sE=s_{-q}\sE$ for $q\le0$ and $s_{-q}f_0\sE$ is zero for $q>0$, we have
\[
E_2^{p,q}(n)(f_0\sE,Y)=\begin{cases}E_2^{p,q}(n)(\sE,Y)&\text{ for }q\le 0\\0&\text{ for }q>0.\end{cases}
\]
Since $\sE\cong \MGL(R^*_\sE)$ by lemma~\ref{lem:LandweberIso}, it follows from \cite[theorem 6.1]{Spitzweck}\footnote{The assumption (SlMGL) in the statement of this result is fulfilled, by the work of Hoyois \cite{Hoyois}.} that
\[
E_2^{p,q}(n)(\sE,Y)=H^{p-q}(Y, R^q_\sE(n-q)).
\]

The computation is now essentially the same as for $\sE=\MGL$, given in proposition~\ref{prop:MGLLWE}: Let $\eta$ be the generic point of some $X\in \Sm/k$. Then in the spectral sequence $E^{**}_*(n)$ converging to $\sE^{*,n}(\eta)$, we have
\[
E_2^{n,n}(n)=R_\sE^n,\quad E_2^{n,n-1}(n)=R_\sE^{n-1}\otimes_\Z k(\eta)^\times
\]
as the only $E_2$ terms contributing to $\sE^{2n,n}(\eta)$ and $\sE^{2n-1,n}(\eta)$, and all differentials entering and leaving these terms are zero. Thus we have isomorphisms $\sE^{2n,n}(\eta)=R^n_\sE$ and $R^{n-1}\otimes k(\eta)^\times\cong \sE^{2n-1,n}(\eta)$.  As the $E_2$ terms for the spectral sequence converging to $(f_0\sE)^{2n,n}(\eta)$ and $(f_0\sE)^{2n-1,n}(\eta)$ are just a truncation of these $E_2$ terms, we have  
\begin{align*}
(f_0\sE)^{2n,n}(\eta)&=\begin{cases} \sE^{2n,n}(\eta)&\text{ for } n\le 0\\0&\text{ for }n>0;\end{cases}\
(f_0\sE)^{2n-1,n}(\eta)=\begin{cases} \sE^{2n-1,n}(\eta)&\text{ for } n\le 1\\0&\text{ for }n>1.\end{cases}
\end{align*}
Arguing as in the proof of proposition~\ref{prop:MGLLWE}, the  result follows.
 \end{proof}

\section{Oriented duality theories}\label{sec:OrientDual} For the remainder of the paper, we will take $S=\Spec  k$, $k$ a field of characteristic zero.

Recall from \cite[\S 1]{LevineOrient} the category $\SP/k$ of {\em smooth pairs} over $k$, with objects $(M,X)$, $M\in \Sm/k$ and $X\subset M$ a closed subset (not necessarily smooth); a morphism $f:(M,X)\to (N,Y)$ is a morphism $f:M\to N$ in $\Sm/k$ such that $f^{-1}(Y)\subset X$. For a full subcategory $\sV$ of $\Sch/k$, let $\sV'$  be the subcategory of $\sV$ with the same objects as $\sV$, but with morphisms the projective morphisms in $\Sch/k$.

Building on work of  Mocanasu \cite{Mocanasu} and Panin \cite{Panin2}, we have defined in \cite[definition 3.1]{LevineOrient} the notion of a bi-graded {\em  oriented duality theory} $(H, A)$ on $\Sch/k$. Here $A$ is a bi-graded oriented cohomology theory on $\SP/k$, $(M,X)\mapsto A_X^{*,*}(M)$, and $H$ is a functor from $\Sch/k'$ to bi-graded abelian groups.  The oriented cohomology theory $A$ satisfies the axioms listed in \cite[definitions 1.2,  1.5]{LevineOrient}. In particular, $(M,X)\mapsto A_X^{*,*}(M)$ admits a long exact sequence 
\[
\ldots\to A_X^{*,*}(M)\to A^{*,*}(M)\to A^{*,*}(M\setminus X)\xrightarrow{\del}A_X^{*+1,*}(M)\to\ldots
\]
where for instance $A^{*,*}(M):=A^{*,*}_M(M)$, and the boundary map $\del$ is part of the data. In addition, there is an excision property and a homotopy invariance property. The ring structure is given by external products and pull-back by the diagonal. The orientation is given by a collection of isomorphisms $\Th^E_X:A_X(M)\to A_X(E)$, for $(M,X)\in \SP/k$ and $E\to M$ a vector bundle, satisfying the axioms of \cite[def. 3.1.1]{Panin2}. We extend some of the results of \cite{Panin2} in \cite[theorem 1.12, corollary 1.13]{LevineOrient} to show that the data of an orientation is equivalent to giving well-behaved push-forward maps $f_*:A_X(M)\to A_Y(N)$ for $(M,X), (N,Y)\in\SP/k$, with the meaning of ``well-behaved" detailed in \cite[\S1]{LevineOrient}.

The homology theory $H$ comes with push-forward maps $f_*:H_{*,*}(X)\to H_{*,*}(Y)$ for $f:X\to Y$ projective, restriction maps $j^*:H_{*,*}(X)\to H_{*,*}(U)$ for each open immersion $j:U\to X$ in $\Sch/k$, external products $\times:H_{*,*}(X)\otimes H_{*,*}(Y)\to H_{*,*}(X\times Y)$,
 boundary maps $\del_{X,Y}:H_{*,*}(X\setminus Y)\to H_{*-1,*}(Y)$ for each closed subset  $Y\subset X$, , isomorphisms $\alpha_{M,X}:H_{*,*}(X)\to A^{2m-*, m-*}_X(M)$ for each $(M,X)\in \SP/k$, $m=\dim_kM$, and finally cap product maps
\[
f^*(-)\cap:A_X^{a,b}(M)\otimes H_{*,*}(Y)\to H_{*-a, *-b}(Y\cap f^{-1}(X))
\]
for $(M,X)\in \SP/k$, $f:Y\to X$ a morphism in $\Sch/k$. These satisfy a number of axioms and compatibilities (see \cite[\S3]{LevineOrient} for details), which essentially say that a structure for $A^{*,*}_X(M)$ is compatible with the corresponding structure for $H_{*,*}(X)$ via the isomorphism $\alpha_{M,X}$. Roughly speaking, this is saying that a particular structure for  $A^{*,*}_X(M)$ depends only on $X$ and not the choice of embedding $X\hookrightarrow M$.

\begin{rem}\label{rem:FGL} Let $L\to Y$ be a line bundle on some $Y\in\Sm/k$ with 0-section $0:Y\to L$. For an  oriented cohomology theory $A$ one has the element
\[
c_1^A(L):=0^*(0_*(1^A_Y)),
\]
where $1^A_Y\in A^0(Y)$ is the unit element.  As pointed out in \cite[corollary 3.3.8]{Panin2}, or as noted in \cite[remark 1.17]{LevineOrient}, for line bundles $L$, $M$ on some $Y\in\Sm/k$,  the elements $c_1(L), c_1(M)\in A^1(Y)$ are nilpotent and commute with one another, hence for each power series $F(u,v)\in A^*(k)[[u,v]]$  the evaluation $F(c_1(L), c_1(M))$ gives a well-defined element of $A^*(Y)$. In addition, the cohomology theory $A$ has a unique associated formal group law $F_A(u,v)\in A^*(k)[[u,v]]$ with
\[
F_A(c^A_1(L), c^A_1(M))=c^A_1(L\otimes M)
\]
for all line bundles $L$, $M$ on $Y\in\Sm/k$.
\end{rem} 

The main example of  oriented duality theory $(H, A)$ is given by an  oriented weak  commutative ring $T$-spectrum  $\sE$ in $\SH(k)$, assuming $k$ is a field admitting resolution of singularities (e.g., characteristic zero), defined by taking
\[
\sE^{a,b}_X(M):=\Hom_{\SH(k)}(\Sigma^\infty_T(M/M\setminus X), \Sigma^{a,b}\sE),
\]
i.e., the usual bi-graded cohomology with supports. For each $X\in\Sch/k$, choose a closed immersion of $X$ into a smooth $M$ and set $\sE'_{a,b}(X):=\sE^{2m-a,m-b}_X(M)$, where $m=\dim_kM$. The fact that $(M,X)\mapsto \sE^{*,*}_X(M)$ defines an oriented bi-graded ring cohomology theory is proved just as in the case of $\sE=\MGL$, which was discussed in \cite[\S 4]{LevineOrient}; the main point is Panin's theorem \cite[theorem 3.7.4]{Panin2}, which says that an orientation for $\sE$  (in the sense of definition~\ref{Def:Orientation}) defines an orientation in the sense of ring cohomology theories for the bi-graded $\sE$-cohomology with supports. The fact that the formula given above for the homology theory $\sE'_{*,*}$ is well-defined and extends to make $(\sE'_{*,*}(-), \sE^{*,*}_{-}(-))$ a bi-graded  oriented duality theory is \cite[theorem 3.4]{LevineOrient}.

\begin{rem} The results of \cite{LevineOrient} were proven in the setting of an oriented commutative ring $T$-spectrum $\sE$, not that of a oriented weak commutative ring $T$-spectrum. However, the constructions and proofs only use values of  $\sE$-cohomology on finite diagrams of smooth $k$-schemes, and thus only rely on identities modulo phantom maps. The arguments thus remain valid in the larger context of oriented weak commutative ring $T$-spectra. We will henceforth make use of the results of \cite{LevineOrient} in this wider context without further comment.
\end{rem}

It follows directly from the construction of $\sE'$ that the assignment $(\sE, c_\sE)\mapsto (\sE', \sE)$ is  functorial in the oriented cohomology theory  $(\sE, c_\sE)$, that is,  let $\ch:(\MGL,c_\MGL)\to(\sE, c_\sE)$ be a morphism of oriented weak commutative ring $T$-spectra. Then $\ch$ extends canonically to a natural transformation of oriented duality theories
\[
(\ch',\ch):(\MGL',\MGL)\to (\sE',\sE).
\]

\section{Algebraic cobordism and oriented duality theories}\label{sec:AlgCobordOrient}
We recall the theory of {\em algebraic cobordism} $X\mapsto \Omega_*(X)$, $X\in \Sch/k$ \cite{LevineMorel}. For each $X\in \Sch/k$, $\Omega_n(X)$ is an abelian group with generators $(f:Y\to X)$, $Y\in \Sm/k$ irreducible of dimension $n$ over $k$ and $f:Y\to X$ a projective morphism \cite[lemma 2.5.11]{LevineMorel}.  $\Omega_*$ is the universal {\em oriented Borel-Moore homology theory} on $\Sch/k$ \cite[theorem 7.1.1]{LevineMorel}, where  an  oriented Borel-Moore homology theory on $\Sch/k$ consists of the data of a functor from $\Sch/k'$ to graded abelian groups, external products, first Chern class operators $\tilde{c}_1(L):\Omega_*(X)\to \Omega_{*-1}(X)$ for $L\to X$ a line bundle, and pull-back maps $g^*:\Omega_*(X)\to \Omega_{*+d}(Y)$ for each \lci morphism $g:Y\to X$ of relative dimension $d$. These of course satisfy a number of compatibilities and additional axioms, see \cite[\S 5.1]{LevineMorel} for details.

For an oriented duality theory $(H,A)$ on $\Sch/k$ and $Y$ in $\Sm/k$ of dimension $d$ over $k$, the {\em fundamental class} $[Y]_{H,A}\in H_d(Y)$ is the image of the unit $1_Y\in A^0(Y)$ under the inverse of the isomorphism $\alpha_Y:H_d(Y)\to A^0(Y)$. For an oriented Borel-Moore homology theory $B$ on $\Sch/k$, we similarly have the fundamental class $[Y]_B\in B_d(Y)$ defined by $[Y]_B:=p^*(1)$, where $1\in B_0(\Spec k)$ is the unit and $p:Y\to \Spec k$ the structure morphism.

We recall the following result from \cite{LevineOrient}:

\begin{prop}[\hbox{\cite[propositions 4.2, 4.4, 4.5]{LevineOrient}}] \label{prop:Universal} Let  $k$ be a field admitting resolution of singularities and let $(H,A)$ be a $\Z$-graded oriented duality theory on $\Sch/k$. \\
1. There is a unique natural transformation $\vartheta_H:\Omega_*\to H_*$ of functors $\Sch/k'\to \GrAb$, such that $\vartheta_H(Y)$ is compatible with fundamental classes for $Y\in\Sm/k$. In addition, $\vartheta_H$ is compatible with pull-back maps for open immersions in $\Sch/k$,  with 1st Chern class operators,  with external products and with cap products. \\
2. For $Y\in\Sm/k$, the map $\vartheta^A(Y):\Omega^*(Y)\to A^*(Y)$ induced by $\vartheta_H$, the identity $\Omega^*(Y)=\Omega_{\dim Y-*}(Y)$ and the isomorphism $\alpha_Y:H_{\dim Y-*}(Y)\to A^*(Y)$  is a ring homomorphism and is compatible with pull-back maps for arbitrary morphisms in $\Sm/k$. Finally, one has
\[
\vartheta^A(Y)(c_1^\Omega(L))=c_1^A(L)
\]
for each line bundle $L\to Y$.
\end{prop}

\begin{rem}\label{rem:CompClassFGL} We have already noted that one has a formal group law $F_A(u,v)\in A^*(k)[[u,v]]$ associated to the oriented cohomology theory $A$. Similarly, for each oriented Borel-Moore homology theory $B$ on $\Sch/k$, there is an associated  formal group law $F_B(u,v)\in B_*(k)[[u,v]]$, determined by the identity 
\[
F_B(c_1(L), c_1(M))=c_1(L\otimes M)
\]
for each pair of line bundles $L,M$ on some $Y\in \Sm/k$ (this follows from \cite[corollary 4.1.8, proposition 5.2.1, proposition 5.2.6]{LevineMorel}).  This gives us the graded ring homomorphism $\phi_B:\L_*\to B_*(k)$ classifying the formal group law $F_B$.

Finally, we recall that the classifying map $\phi_\Omega:\L_*\to \Omega_*(k)$ is an isomorphism \cite[theorem 1.2.7]{LevineMorel}.
\end{rem}

The following is a direct consequence of proposition~\ref{prop:Universal}.

\begin{cor} \label{cor:NatTrans} Let $(\sE,c_\sE)$ be pair consisting of a weak commutative ring $T$-spectrum  $\sE\in \SH(k)$ with  orientation  $c$, and let $(\sE'_{*,*}, \sE^{*,*})$ be the corresponding bi-graded oriented duality theory. There is a  unique natural transformation 
\[
\vartheta_{(\sE,c)}:\Omega_*\to \sE'_{2*,*}
\]
of functors $\Sch/k'\to \GrAb$, such that $\vartheta_{(\sE,c)}(Y)$ is compatible with fundamental classes for $Y\in\Sm/k$. In addition, $\vartheta_{(\sE,c_\sE)}$ is compatible with pull-back maps for open immersions in $\Sch/k$, 1st Chern class operators, external products and cap products. For $Y\in\Sm/k$, the map $\vartheta^\sE(Y):\Omega^*(Y)\to \sE^{2*,*}(Y)$ induced by $\vartheta_{(\sE,c_\sE)}$ is a ring homomorphism and is compatible with pull-back maps for arbitrary morphisms in $\Sm/k$, and satisfies
\[
\vartheta_{(\sE,c)}(Y)(c_1^\Omega(L))=c_1^\sE(L)
\]
for each line bundle $L\to Y$.
\end{cor}

\begin{rem}\label{rem:Char}  By \cite[lemma 2.5.11]{LevineMorel}, $\Omega_*(X)$ is generated as an abelian group by the cobordism cycles $(f:Y\to X)$, $Y\in\Sm/k$ irreducible, $f:Y\to X$  a projective morphism.  Furthermore,  the identity $(f:Y\to X)=f_*([Y]_\Omega)$ holds in $\Omega_{\dim Y}(X)$. Thus $\vartheta_{(\sE,c_\sE)}$ is characterized by the formula
\[
\vartheta_{(\sE,c_\sE)}(f:Y\to X):=f_*^{\sE'}([Y]_{\sE',\sE}).
\]
\end{rem}

We may apply corollary~\ref{cor:NatTrans}  in the universal case: $\sE=\MGL$ with its canonical orientation. This gives us the natural transformation
\begin{equation}\label{eqn:NatMGL}
\vartheta_\MGL:\Omega_*\to \MGL'_{2*,*}.
\end{equation}

\begin{thm}[\hbox{\cite[theorem 3.1]{LevineComparison}}] \label{thm:MGLComp} If $k$ is a field of characteristic zero, then the natural transformation \eqref{eqn:NatMGL} is an isomorphism.
\end{thm}
This result relies on the Hopkins-Morel spectral sequence, see \cite{HopkinsMorel, Hoyois}.

In the course of the proof of theorem~\ref{thm:MGLComp}, we proved another result which we will be using here. 

Let $X$ be in $\Sch/k$ and let $d=d_X:=\max_{X'}\dim_kX'$, as $X'$ runs over the irreducible components of $X$. We define $\MGL_{2*,*}^{\prime (1)}(X)$ by
\[
\MGL_{2*,*}^{\prime (1)}(X):=\colim_W \MGL_{2*,*}'(W)
\]
as $W$ runs over all (reduced) closed subschemes of $X$ which contain no dimension $d$ generic point of $X$; $\Omega_*^{(1)}(X)$ is defined similarly. The natural transformation $\vartheta_\MGL$ gives rise to the commutative diagram
\begin{equation}\label{eqn:HM1}
\xymatrix{
\Omega_*^{(1)}(X)\ar[d]_{\vartheta^{(1)}}\ar[r]^-{i_*}&\Omega_*(X)\ar[d]_{\vartheta(X)}\ar[r]^-{j^*}&\oplus_{\eta\in X_{(d)}}\Omega_*(k(\eta))\ar[r]\ar[d]^{\vartheta}&0\\
\MGL_{2*,*}^{\prime (1)}(X)\ar[r]_-{i_*}&\MGL_{2*,*}'(X)\ar[r]_-{j^*}&\oplus_{\eta\in X_{(d)}}\MGL_{2*,*}'(k(\eta))\ar[r]&0}
\end{equation}
with exact rows and with all vertical arrows isomorphisms. As $(\MGL', \MGL)$ is an oriented duality theory, the bottom line extends to the long exact sequence
\begin{multline*}
\ldots\to \oplus_{\eta\in X_{(d)}} \MGL_{2*+1,*}'(k(\eta))\xrightarrow{\del} \MGL_{2*,*}^{\prime (1)}(X)\\\xrightarrow{i_*}\MGL_{2*,*}'(X)\xrightarrow{j^*}\oplus_{\eta\in X_{(d)}}\MGL_{2*,*}'(k(\eta))\to0.
\end{multline*}
By proposition~\ref{prop:MGLLWE}(2), the map
\[
t_\MGL(\eta):\L_{*-d+1}\otimes k(\eta)^\times\to  \MGL_{2*+1,*}'(k(\eta))
\]
 an isomorphism for each $\eta\in X_{(d)}$. We have constructed in \cite[\S6]{LevineComparison} a group homomorphism
\[
\Div:\L_{*-d+1}\otimes\oplus_{\eta\in X_{(d)}} \Z[k(\eta)^\times]\to \Omega_*^{\prime(1)}(X)
\]
with $\vartheta^{(1)}\circ\Div=\del\circ\oplus_\eta t_\MGL(\eta)$. Since the maps $\vartheta^{(1)}$ and $\vartheta(X)$ are isomorphisms, the map $\Div$ factors through the surjection
\[
\L_{*-d+1}\otimes\oplus_{\eta\in X_{(d)}} \Z[k(\eta)^\times]\to \L_{*-d+1}\otimes\oplus_{\eta\in X_{(d)}}k(\eta)^\times,
\]
and we have the exact sequence
\begin{equation}\label{eqn:PresentOmega}
\oplus_{\eta\in X_{(d)}} \L_{*-d+1}\otimes k(\eta)^\times \xrightarrow{\Div} 
 \Omega_{*}^{\prime(1)}(X)
 \xrightarrow{i_*}\Omega_*(X)\xrightarrow{j^*}\oplus_{\eta\in X_{(d)}} \Omega_*(k(\eta))\to 0
 \end{equation}
 and the extension of diagram \eqref{eqn:HM1} to the commutative diagram 
 \begin{equation}\label{eqn:HM2}
\xymatrixcolsep{14pt}
\xymatrix{
\oplus_{\eta} \L_{*-d+1}\otimes k(\eta)^\times\ar[r]^-{\Div}\ar@{=}[d]&\Omega_*^{(1)}(X)\ar[d]_{\vartheta^{(1)}}\ar[r]^-{i_*}&\Omega_*(X)\ar[d]_{\vartheta(X)}\ar[r]^-{j^*}&\oplus_{\eta}\Omega_*(k(\eta))\ar[r]\ar[d]^{\vartheta}&0\\
\oplus_{\eta} \L_{*-d+1}\otimes k(\eta)^\times\ar[r]_-{div_\MGL}&\MGL_{2*,*}^{\prime (1)}(X)\ar[r]_-{i_*}&\MGL_{2*,*}'(X)\ar[r]_-{j^*}&\oplus_{\eta}\MGL_{2*,*}'(k(\eta))\ar[r]&0}
\end{equation}
with exact rows and vertical arrows isomorphisms. Here $div_\MGL:=\del\circ\oplus_\eta t_\MGL(\eta)$.

\section{The comparison map}  \label{sec:Comp}  Let $(\sE,c)$ be  a weak commutative ring $T$-spectrum in $\SH(k)$ with  orientation  $c$, and let $(\sE'_{*,*}, \sE^{*,*})$ be the corresponding bi-graded oriented duality theory.  We have the natural transformation
\[
\vartheta_{(\sE,c)}:\Omega_*\to \sE'_{2*,*}
\]
given by corollary~\ref{cor:NatTrans}.  The map $\vartheta_{(\sE,c)}(k)$ makes $R^\sE_*$ an $\Omega_*$-algebra; we let $\Omega_*^\sE$ be the oriented Borel-Moore homology theory $\Omega_*^\sE(X):=R^\sE_*\otimes_{\Omega_*(k)}\Omega_*(X)$.

As the external products make $\sE'_{2*,*}(X)$ an $R^\sE_*$-module and the maps $f_*$, $\tilde{c}_1(L)$ are $R^\sE_*$-module maps, we see that  $\vartheta_{(\sE,c)}$ descends to a natural transformation
\[
\bar\vartheta_{(\sE,c)}:\Omega_*^\sE\to \sE'_{2*,*}.
\]

\begin{lem} \label{lem:LandExact} Suppose that the oriented weak commutative ring $T$-spectrum $(\sE,c)$ is geometrically Landweber exact. Then\\
1. For $X\in \Sch/k$ and $\eta\in X$ a point, the map $\bar\vartheta_\sE:\Omega^\sE_*(\eta)\to \sE'_{2*,*}(\eta)$ is an isomorphism.\\
2. Take $X$ in $\Sch/k$ and let $j_i:\eta_i\to X$, $i=1,\ldots, r$ be all the  generic points of $X$. Then the restriction map $j^*:\sE'_{2*,*}(X)\to \oplus_{i=1}^r\sE'_{2*,*}(\eta_i)$ is surjective.\\
3. For each generic point $\eta$ of $X$, the map $t_\sE(\eta):\sE^{2*,*}(k)\otimes_\Z k(\eta)^\times\to \sE^{2*+1,*+1}(\eta)$ is surjective.
\end{lem}

\begin{proof} For (1), we may replace $X$ with the closure of $\eta$ in $X$, so we may assume that $\eta$ is the generic point of $X$.  We note that $\vartheta_\sE=\phi_\sE\circ \vartheta_\MGL$. In addition $\vartheta_\MGL$ is an isomorphism, so $\bar\vartheta_\sE(X)$ is up to this isomorphism the same as the map $\MGL'_{2*,*}(X)\otimes_{\L_*}R^\sE_*\to \sE'_{2*,*}(X)$. (1) then follows from the hypothesis on $\sE$.

For (2), we have the commutative diagram
\[
\xymatrix{
\Omega_*(X)\ar[d]_{\vartheta(X)}\ar[r]^-{j_\Omega^*}& \oplus_{\eta\in X_{(d)}}\Omega_*(\eta)\ar[d]^{\oplus_\eta\vartheta}\\
\sE'_{2*,*}(X)\ar[r]_-{j_\sE^*}&\oplus_{\eta\in X_{(d)}}\sE'_{2*,*}(\eta)}
\]
By (1), $\vartheta(\eta):\Omega_*(\eta)\to \sE'_{2*,*}(\eta)$ is surjective. The map $j^*_\Omega$  is also surjective, using the right exact localization sequence for $\Omega_*$. Thus, the map $j^*_\sE$ is also surjective.

For (3), we can rewrite the map  $t_\sE(\eta)$ as the composition
\begin{multline*}
R_\sE^*\otimes_{\L^*}\MGL^{2*,*}(\eta)\otimes \MGL^{1,1}(\eta)\xrightarrow{\cup}R_\sE^*\otimes_{\L^*}\MGL^{2*+1, *+1}(\eta)
\\\xrightarrow{\id\otimes\phi_\sE}R_\sE^*\otimes_{\L^*}\sE^{2*+1, *+1}(\eta)\cong \sE^{2*+1, *+1}(\eta).
\end{multline*}
As the map  $\MGL^{2*,*}(\eta)\otimes \MGL^{1,1}(\eta)\to \MGL^{2*+1, *+1}(\eta)$ is surjective, and  the map $ \MGL^{2*+1, *+1}(\eta)\to \sE^{2*+1, *+1}(\eta)$ is surjective by hypothesis,  it follows that $t_\sE(\eta)$  is also surjective. 
\end{proof}

\begin{thm} \label{thm:GeomLandExact} Suppose that a oriented weak commutative ring $T$-spectrum $(\sE,c)$  is geometrically Landweber exact. Then the natural transformation $\bar\vartheta_{(\sE,c)}:\Omega_*^\sE\to \sE'_{2*,*}$ is an  isomorphism.
\end{thm}

\begin{proof}  We write $\bar\vartheta$ for $\bar\vartheta_{(\sE,c)}$.    For $\eta$ a   point of $X$, the map 
$\bar\vartheta(\eta):\Omega_*^\sE(\eta)\to \sE'_{2*,*}(\eta)$ is an isomorphism by lemma~\ref{lem:LandExact}(1).  In particular, if $X$ has dimension zero over $k$, then $\bar\vartheta(X)$ is an isomorphism. 

We proceed by induction on the maximum $d$ of the dimensions of the components of $X$; we may assume that $X$ is reduced.
We use the constructions and notations from theorem~\ref{thm:MGLComp} and the discussion following that theorem.  We let $\sE_{2*,*}^{\prime(1)}(X)$ be the inductive limit
\[
\sE_{2*,*}^{\prime (1)}(X):=\colim_W \sE'_{2*,*}(W)
\]
as $W$ runs over all (reduced) closed subschemes of $X$ which contain no dimension $d$ generic point of $X$. This, together with the map $\Div$ defined following theorem~\ref{thm:MGLComp}, and the exact localization  sequence for $\sE'_{*,*}$ gives us the commutative diagram with exact rows
\begin{equation}\label{eqn:Diag1}
\xymatrixcolsep{10pt}
\xymatrix{
\oplus_{\eta\in X_{(d)}}\L_{*-d+1}\otimes k(\eta)^\times\ar[r]^-{\Div}&\Omega_*^{(1)}(X)\ar[d]_{\vartheta^{(1)}}\ar[r]^-{i_*}&\Omega_*(X)\ar[d]_{\vartheta(X)}\ar[r]^-{j^*}&\oplus_{\eta\in X_{(d)}}\Omega_*(\eta)\ar[r]\ar[d]_{\vartheta}&0\\
\oplus_{\eta\in X_{(d)}}\sE_{2*+1,*}'(k(\eta))\ar[r]_-{\partial} &\sE_{2*,*}^{\prime (1)}(X)\ar[r]_-{i_*}&\sE'_{2*,*}(X)\ar[r]_-{j^*}&\oplus_{\eta\in X_{(d)}}\sE'_{2*,*}(\eta)\ar[r]&0;}
\end{equation}
the surjectivity in the bottom row comes from lemma~\ref{lem:LandExact}(2).

We apply $R^\sE_*\otimes_{\Omega_*(k)}(-)$ to the top row in \eqref{eqn:Diag1}. As noted at the beginning of this section,  the vertical maps in \eqref{eqn:Diag1}  descend to give the commutative diagram
\begin{equation}\label{eqn:Diag2}
\xymatrixcolsep{14pt}
\xymatrix{
\oplus_{\eta\in X_{(d)}}R^\sE_{*-d+1}\otimes k(\eta)^\times\ar[r]^-{\Div_\sE}&\Omega_*^{\sE(1)}(X)\ar[d]_{\bar\vartheta^{(1)}}\ar[r]^-{i_*}&\Omega^\sE_*(X)\ar[d]_{\bar\vartheta(X)}\ar[r]^-{j^*}&\oplus_{\eta\in X_{(d)}}\Omega^\sE_*(\eta)\ar[r]\ar[d]_{\bar\vartheta}&0\\
\oplus_{\eta\in X_{(d)}}\sE_{2*+1,*}'(k(\eta))\ar[r]_-{\partial} &\sE_{2*,*}^{\prime (1)}(X)\ar[r]_-{i_*}&\sE_{2*,*}'(X)\ar[r]_-{j^*}&\oplus_{\eta\in X_{(d)}}\sE_{2*,*}'(\eta)\ar[r]&0;}
\end{equation}

By induction on $d$, the map $\vartheta^{(1)}$ is an isomorphism; we have already seen that $\bar\vartheta$ is an isomorphism. We note that the bottom row is a sequence of $R^\sE_*$-modules via the the $\sE'_{2*,*}(k)$-module structure given by external products. 

The universal property of $\MGL$ gives the canonical morphism of oriented weak commutative ring $T$-spectra
\[
\phi_{\sE}:(\MGL, c_\MGL)\to (\sE, c),
\]
which extends to the map of bi-graded oriented duality theories
\[
\phi_{\sE}:(\MGL'_{*,*}, \MGL^{*,*})\to (\sE'_{*,*}, \sE^{*,*}).
\]

As discussed in remark~\ref{rem:Unit}, we have for each orientable $\sE$ and each $Y\in \Sm/k$  the $R_\sE^*$-module map
\[
t_{\sE}(Y):\sO_Y^\times(Y)\otimes R_\sE^*\to \sE^{2*+1,*+1}(Y),
\]
natural in $\sE$ and $Y$.  This gives us the commutative diagram
\begin{equation}\label{eqn:Diag5}
\xymatrix{
\oplus_{\eta\in X_{(d)}} \L_{*-d+1}\otimes k(\eta)^\times\ar[r]^-{t_\MGL}\ar[d]_\pi&
\oplus_{\eta\in X_{(d)}}\MGL'_{2*+1,*}(k(\eta))\ar[r]^-{\partial}\ar[d]_{\phi_{\sE}}
&\MGL_{2*,*}^{\prime (1)}(X)\ar[d]_{\phi_{\sE}(X)^{(1)}}\\
\oplus_{\eta\in X_{(d)}} R^\sE_{*-d+1}\otimes k(\eta)^\times\ar[r]_-{t_{\sE}}&\oplus_{\eta\in X_{(d)}}\sE_{2*+1,*}'(k(\eta))\ar[r]_-{\partial}&\sE_{2*,*}^{\prime (1)}(X)
}
\end{equation}
where $\pi$ is induced by the classifying map $\L_*\to R^\sE_*$.

Take $\eta\in X_{(d)}$.   By lemma~\ref{lem:LandExact}(3) the map
\[
t_{\sE}(\eta):R^\sE_{*-d+1}\otimes k(\eta)^\times\to \sE_{2*+1,*}'(\eta).
\]
is surjective for each $\eta$. Defining $div_\sE=\partial\circ \sum_\eta t_{\sE}(\eta)$ and putting this into the diagram \eqref{eqn:Diag2}  gives us the commutative diagram
\begin{equation}\label{eqn:Diag3}
\xymatrixcolsep{14pt}
\xymatrix{
\oplus_{\eta\in X_{(d)}}R^\sE_{*-d+1}\otimes k(\eta)^\times\ar[r]^-{\Div_\sE}&\Omega_*^{\sE(1)}(X)\ar[d]_{\bar\vartheta^{(1)}}\ar[r]^-{i_*}&\Omega^\sE_*(X)\ar[d]_{\bar\vartheta(X)}\ar[r]^-{j^*}&\oplus_{\eta\in X_{(d)}}\Omega^\sE_*(\eta)\ar[r]\ar[d]_{\bar\vartheta}&0\\
\oplus_{\eta\in X_{(d)}}R^\sE_{*-d+1}\otimes k(\eta)^\times\ar[r]_-{div_\sE} &\sE_{2*,*}^{\prime (1)}(X)\ar[r]_-{i_*}&\sE_{2*,*}'(X)\ar[r]_-{j^*}&\oplus_{\eta\in X_{(d)}}\sE_{2*,*}'(\eta)\ar[r]&0}
\end{equation}
with the bottom row exact and the top row a complex. Recalling that 
 $div_\MGL=\partial\circ \sum_\eta t_{\MGL}(\eta)$,  the commutativity of diagram \eqref{eqn:Diag5} gives us the commutative diagram
\begin{equation}\label{eqn:Diag4}
\xymatrix{
\oplus_{\eta\in X_{(d)}} \L_{*-d+1}\otimes k(\eta)^\times\ar[r]^-{div_\MGL}\ar[d]_\pi
&\MGL_{2*,*}^{\prime (1)}(X)\ar[d]_{\phi_{\sE}(X)^{(1)}}\\
\oplus_{\eta\in X_{(d)}} R^\sE_{*-d+1}\otimes k(\eta)^\times\ar[r]_-{div_{\sE}}&\sE_{2*,*}^{\prime (1)}(X).
}
\end{equation}

We claim the identity map on $\oplus_\eta R^\sE_{*-d+1}\otimes k(\eta)^\times$ fills in the diagram \eqref{eqn:Diag3} to a commutative diagram. Assuming this claim, it follows by a diagram chase that the top row is exact and the map $\bar\vartheta(X)$ is an isomorphism. 

To prove the claim,  it follows from the characterization of $\vartheta_\MGL$ and $\vartheta_{\sE}$ given in remark~\ref{rem:Char} that $\vartheta_{\sE}=\phi_{\sE}\circ\vartheta_\MGL$. Thus patching diagram \eqref{eqn:Diag4} into the left-hand square in the commutative diagram \eqref{eqn:HM2} yields the commutative diagram
\begin{equation}\label{eqn:HM3}
\xymatrixcolsep{25pt}
\xymatrix{
\oplus_{\eta\in X_{(d)}} \L_{*-d+1}\otimes k(\eta)^\times\ar[r]^-{\Div}\ar@{=}[d]&\Omega_*^{(1)}(X)\ar[d]_{\vartheta_\MGL^{(1)}}\ar@/^40pt/[dd]^{\vartheta_\sE^{(1)}}\\
\oplus_{\eta\in X_{(d)}} \L_{*-d+1}\otimes k(\eta)^\times\ar[r]^-{div_\MGL}\ar[d]_\pi
&\MGL_{2*,*}^{\prime (1)}(X)\ar[d]_{\phi_{\sE}(X)^{(1)}}\\
\oplus_{\eta\in X_{(d)}} R^\sE_{*-d+1}\otimes k(\eta)^\times\ar[r]_-{div_{\sE}}&\sE_{2*,*}^{\prime (1)}(X)
}
\end{equation}
As $\Div_{\sE}:R^\sE_*\otimes k(\eta)^\times\to \Omega_*^{\sE(1)}(X)$ is just the map formed by applying the functor $R^\sE_*\otimes_{\L}(-)$ to $\Div: \L\otimes k(\eta)^\times\to\Omega_*^{(1)}(X)$, the desired commutativity follows from the commutativity of the outer square in \eqref{eqn:HM3}.
\end{proof}

\begin{cor}\label{cor:main} Let $(\sE,c)$ be a Landweber exact oriented weak ring $T$-spectrum in $\SH(k)$, $k$ a field of characteristic zero and let $(f_0\sE, c_0)$ be the effective cover of $(\sE,c)$. Then the canonical natural transformations of oriented  Borel-Moore homology theories on $\Sch_k$
\begin{align*}
&\vartheta_{\sE, c}:R_\sE\otimes_\L\Omega_*\to \sE'_{2*,*}\\
&\vartheta_{f_0\sE, c_0}:R_{f_0\sE}\otimes_\L\Omega_*\to f_0\sE'_{2*,*}\\
\end{align*}
are isomorphisms. Moreover,  the canonical natural transformations of oriented  cohomology theories on $\Sm/k$
\begin{align*}
&\vartheta_{\sE, c}:R_\sE\otimes_\L\Omega^*\to \sE^{2*,*}\\
&\vartheta_{f_0\sE, c_0}:R_{f_0\sE}\otimes_\L\Omega^*\to f_0\sE^{2*,*}\\
\end{align*}
are isomorphisms.
\end{cor}

\begin{proof}
By theorem~\ref{thm:LexactGLExact}, both $(\sE,c)$ and $(f_0\sE, c_0)$ are geometrically Landweber exact. We then apply theorem~\ref{thm:GeomLandExact} to yield the desired isomorphisms of oriented Borel-Moore homology theories. 

The statement about the oriented cohomology theories on $\Sm/k$ follows by restriction from $\Sch/k$ to $\Sm/k$, using the equivalence of oriented Borel-Moore homology theories and oriented cohomology theories on $\Sm/k$ \cite[proposition 5.2.1]{LevineMorel}.
\end{proof}

\end{document}